\documentclass[english,11pt]{amsart}
\usepackage[T1]{fontenc}
\usepackage[latin9]{inputenc}
\usepackage{geometry}
\usepackage{color}
\usepackage{mathtools}
\usepackage{amsmath}
\usepackage{amsthm}
\usepackage{amssymb}
\usepackage{esint}
\usepackage{float}
\usepackage{dcolumn}
\usepackage{graphicx}

\theoremstyle{plain}
\newtheorem{theorem}{Theorem}
\newtheorem{claim}[theorem]{Claim}
\newtheorem{proposition}[theorem]{Proposition}
\newtheorem{corollary}[theorem]{Corollary}
\newtheorem{lemma}[theorem]{Lemma}

\newtheorem{remark}{Remark}
\usepackage{comment}
\usepackage{subfig}
\usepackage{booktabs}
\usepackage[colorlinks=true, allcolors=blue]{hyperref}

\makeatletter
\providecommand{\tabularnewline}{\\}
\numberwithin{figure}{section}

\providecommand{\tabularnewline}{\\}
\floatstyle{ruled}
\newfloat{algorithm}{tbp}{loa}
\providecommand{\algorithmname}{Algorithm}
\floatname{algorithm}{\protect\algorithmname}
\usepackage{algorithm,algpseudocode}

\definecolor{mygreen}{rgb}{0.1,0.75,0.2}

\providecommand{\bbs}[1]{\left(#1\right)}

\DeclareMathOperator{\st}{s.t.~}

\newcommand{\pt}{\partial}

\newcommand{\ud}{\,\mathrm{d}}

\newcommand{\tL}{\tilde{L}}

\makeatother

\numberwithin{equation}{section}
\usepackage{babel}
\begin{document}
\title{Mean Field Games for Controlling Coherent Structures in Nonlinear Fluid Systems}
\author{Yuan Gao\,\,  and \,\,  Di Qi}
\address{Department of Mathematics, Purdue University, West Lafayette, IN}
\email{gao662@purdue.edu, qidi@purdue.edu}

\maketitle
\begin{abstract}
This paper discusses the control of coherent structures in turbulent flows, which has broad applications among complex systems in science and technology.
Mean field games have been proved a powerful tool and are proposed here to control the stochastic Lagrangian tracers as players tracking the flow field.
We  derive optimal control solutions for general nonlinear fluid  systems using mean field game models, and develop computational algorithms to efficiently solve the resulting coupled forward and backward mean field system. A precise link is established for the control of Lagrangian tracers and the scalar vorticity field based on the functional Hamilton-Jacobi equations derived from the mean field models. New iterative numerical strategy is then constructed to compute the optimal solution with fast convergence.
We verify the skill of the mean field control models and illustrate their practical efficiency on a prototype model modified from the viscous Burger's equation under various cost functions in both deterministic and stochastic formulations. The good model performance implies potential effectiveness of the strategy for more general high-dimensional turbulent systems.
\end{abstract}


%

\section{Introduction and background}
Control of  complex fluid systems characterized by a wide multiscale spectrum and nonlinear coupling between different scales remains a grand challenge with crucial applications among many fields of science and engineering \cite{brunton2015closed,macmartin2014dynamics,kim2007linear,anderson2007optimal}.
Flow fields undergoing turbulent motions often demonstrate large-scale coherent structures, and control of the important transporting behaviors 
requires dealing with the coherent structures demonstrating strong nonlinear interactions among multiple scales \cite{majda2017effective,kooloth2021coherent,covington2023effective}. In addition, the Eulerian flow field can be captured by Lagrangian tracers advected by the background flow. Control of the multiscale flows through the  Lagrangian tracers sets a useful alternative approach in many practical scenarios \cite{majda1999simplified,apte2008bayesian,garcia2022lagrangian,chen2023launching}.  New precise theoretical analysis and effective control strategies are still needed considering the nonlinear dynamics and high computational demand.

In recent years, mean field game (MFG) theory was proposed independently by \textsc{Larsy-Lions} \cite{lasry2007mean} and \textsc{Caines-Huang-Malhame} \cite{huang2006large} to study a  large game  system with identical players using the dynamics of their  population/distribution. Indistinguishable individual player takes strategy/control based on the population states (without observing all the  strategies of other players), meanwhile the distribution of individuals yields the population dynamics. The equilibrium, also known as the Nash system, is solved  by a \emph{coupled MFG system}   consisting of a forward Fokker-Planck equation (describing the evolution of the density/population of the individuals) and a backward Hamilton-Jacobi equation (describing the control/strategy of each indistinguished player); cf. \cite{achdou2021mean}. Due to this simplification for large game system, MFG   has been proved to be a powerful tool
to study  the   equilibrium behavior of infinitely many weakly-interacting players \cite{cardaliaguet2019master,cousin2011mean}. The associated mean-field control problem also stimulates many applications in swarm drones, generative models, transition path theory and mathematical finance \cite{carmona2018probabilistic, li2022computational, gao2023transition, zhang2023mean}.

In this paper, we propose to address the challenging problem of controlling nonlinear  transport of coherent structures by developing models of mean field games. 
First, we introduce the mean field game model for the control of \emph{passive scalar tracers}. The stochastic motion of the Lagrangian tracers, which are immersed in the fluid flow field and passively advected by the flow velocity, can thus be viewed as the large number of identical players in the mean field games. Therefore, the dynamical behavior of the tracer density and the optimal control on each of the identical Lagrangian tracers can be solved by the corresponding MFG system with uncoupled forward and backward equations. 
Next, we design control models for the transport of representative flow structures in the advection and diffusion of a \emph{scalar vorticity field} under a general formulation. The Lagrangian tracers as players tracking the fluid vorticity field then become nonlocally coupled by their accumulated density function. The evolution continuum density of the tracers shows to coincide with the vorticity equation recovered by the McKean-Vlasov equation \cite{mckean1966class}. Thus, the flow solutions can be inversely captured by the empirical distribution of the large number of Lagrangian tracers (players). An MFG system of closely coupled forward and backward equations are derived for the optimal solution of the flow control problem.
In addition, the corresponding functional Hamilton-Jacobi equations are derived describing the evolution of the value function in both the tracer and flow control models. This leads to an interesting link between the uncoupled tracer  MFG system and the coupled vorticity MFG system in a unified way.

The nonlinear nature of vorticity equation makes the corresponding Hamiltonian non-separable and thus new decouple procedures are needed.
A series of numerical strategies have been developed for solving the coupled MFG equations \cite{achdou2010mean,lauriere2021numerical}. Methods based on a fixed-point iteration \cite{carlini2014fully,gueant2012mean} have been used to decouple the forward and backward equations with semi-Lagrangian schemes and  the Cole-Hopf transformation to convert to linear equations. One strategy named fictitious play \cite{cardaliaguet2017learning,inoue2023fictitious} is shown to help the convergence of the fixed-point iteration using the average of the entire history of the outputs. Another approach uses the primal-dual hybrid gradient method \cite{chambolle2011first,li2022controlling} to solve the control problem as a saddle point problem.  Usually, these methods require transformation of the original equations in the first step and may require a large number of iterations in numerical simulations.

Based on the typical structure of the nonlinear flow model, we propose a new strategy to efficiently find the optimal solution of the coupled MFG system without any special treatment to the original equations.  The numerical method is based on the link developed for the uncoupled tracer MFG system and the coupled flow MFG equations, where the optimal solution for the coupled equations becomes a fixed point of the solution map for the uncoupled problem. This inspires an effective iterative approach by first solving the backward equation given the output solution from the previous step, then using the optimal control function to solve the density function from the forward equation. 
Stability and fast convergence of the proposed iterative scheme are further guaranteed by a straightforward interpolation step  during each iteration. We show the necessity of adopting this essential interpolation step to effectively reduce the cost function value in each iteration cycle. This step also shows to enable fast convergence during the iterating updates. 
The model performance and the effectiveness of the proposed numerical algorithm are tested on a representative prototype model adapted from the viscous Burger's equation simulating the multiscale vortical advection by an incompressible flow field. Using the new iterative approach, fast convergence to the optimal solution is observed within a few steps under different loss function structures and in both one- and two-dimensional cases.

The paper is organized as follows. A description for the tracer and flow transport equations in both deterministic and stochastic formulation is introduced in Section~\ref{sec:formu}. Mean field models as well as the associated systems for the control problems are constructed in Section~\ref{sec:control_models}. The mean field models can be written in Hamilton dynamics and a link between the tracer and flow control models is built in Section~\ref{sec:hje}.  Effective computational schemes corresponding to the theoretical models are then proposed in Section~\ref{sec:comp_strategy} with numerical confirmation on prototype models in Section~\ref{sec:numerics}. We conclude the paper with a summarizing discussion in Section~\ref{sec:summary}.

\section{Formulation of the fluid control equations}\label{sec:formu}
We start with the basic setups and notations for the flow system and the associated control problem. In particular, the Eulerian flow system has an equivalent stochastic Lagrangian formulation based on which the mean field games will be developed. 

\subsection{Governing equation for the transport of a scalar field}
Let $u\left(x,t\right)$ be the incompressible velocity field of
a fluid as a function on the Eulerian spatial coordinates $x\in\mathbb{T}^{d}$ 
 and time $t$. The fluid field can be modeled by the evolution of a \emph{scalar vorticity
field} $q\left(x,t\right)$  according to the standard advection-diffusion
equation 
\begin{equation}
\begin{aligned}
\frac{\partial q}{\partial t}+u\cdot\nabla q  =D\left(\Delta\right) q+f\left(x,t\right), \quad 
u\left(x,t\right) = \mathcal{T}q\left(x,t\right), \quad q\left(x,0\right)=q_{0}\left(x\right).\label{eq:vort_eqn}
\end{aligned}
\end{equation}
On the right hand side of the equation, the flow is  
subject to external forcing effect represented by $f\left(x,t\right)$
and generalized damping and dissipation effects $D\left(\Delta\right)$. 
The velocity field satisfies $\nabla\cdot u=0$ due to incompressibility. Therefore,  the vorticity field $q\left(x,t\right)$ is uniquely linked to the velocity field by the (invertible) linear operator, $u=\mathcal{T}q$, so that the flow solution is fully represented by the scalar vorticity transported by the flow velocity.
The scalar model can serve as a desirable general framework for developing theories and practical methods for turbulent transport and control of coherent structures \cite{majda2003introduction,pedlosky2013geophysical}. 
Applications of the general fluid advection-diffusion model \eqref{eq:vort_eqn} include wide classes  of turbulent flows in  realistic natural and engineering systems such as the passive tracer diffusion \cite{majda1999simplified}, turbulent transport in geophysical turbulence \cite{pedlosky2013geophysical,qi2018predicting}, and the anomalous transport in fusion plasmas \cite{Diamond2010,qi2019transient}. 

It is noticed that the above equation \eqref{eq:vort_eqn} usually permits a set of metastable steady state
solutions \cite{qi2018rigorous}. The steady states provide a natural characterization
of the `coherent structures' observed in the flow field. In addition, turbulent features will be developed in the solution of the vorticity flow \eqref{eq:vort_eqn}  (such as cascading small-scale fluctuating
vortices and strongly mixing structures) due to the inherent instability
and nonlinear coupling in the governing equation. Our attention will
be given to finding the `optimal' fluid transition undergoing some
transient motions from one 
initial state $Q_{i}$ to another final steady state
$Q_{f}$  by imposing `proper' external
forcing effect.

\begin{remark}
One motivating example of \eqref{eq:vort_eqn}  comes from the two-dimensional geophysical turbulence, where the potential vorticity $q=Z\left(\psi\right)$ is associated
with the flow streamfunction $\psi$  via a linear operator $Z$ and may include a variety of  physical
effects (for example,  $Z\left(\psi\right)=\Delta\psi+\beta y+h\left(x\right)$ including the rotation $\beta$ and the topography $h$ \cite{majda2006nonlinear}).
Then the corresponding velocity field can be determined by the stream function as $u=\nabla^{\bot}\psi=\left(-\partial_y \psi,\partial_x \psi\right)$ through a linear operator.
\end{remark}
\begin{remark}
The vorticity model can be generalized to three dimensions by introducing a multi-layer structure in the $z$-direction or using proper potential vorticity invertibility approximations \cite{pedlosky2013geophysical}. Still, we restrain ourselves in the simpler $d\leq 2$ cases for deriving the general modeling systems under the cleaner setup.
\end{remark}

\subsection{Stochastic Lagrangian equations for tracer transport}
From a stochastic perspective, the continuity equation of the flow field \eqref{eq:vort_eqn}
can be captured by the collective performance of a group of \emph{Lagrangian particles/tracers} $X_{s}$
transported by the flow velocity field. In addition, we introduce a \emph{control field} $\alpha_s$ as the additional driven forcing exerted on the flow/tracer field.
In the mean-field game context,  those particles are also called indistinguished individual players.

We introduce the fluid state (i.e., the population/vorticity) from the solution of the vorticity equation \eqref{eq:vort_eqn} as $q_s(\cdot),t\leq s\leq T$. The velocity field is then denoted as $u_s\left(x;q_s\right)=\mathcal{T}q_s\left(x\right)$ explicitly indicating the direct relation with $q_s$. Suppose for now that the control field $\alpha_s(x;q_s)$ is given dependent on the fluid motion, which will be determined by an optimization procedure as described next in Section~\ref{sec:control_models}.  Then the
equation for the trajectory of each of the particle is described by the stochastic differential equation (SDE)
\begin{equation}\label{eq:ctrl_sde}
\begin{aligned}
\mathrm{d}X_{s}  = \left[u_{s}\left(X_{s};q_{s}\right)+\alpha_{s}\left(X_{s};q_{s}\right)\right]\mathrm{d}s+\sqrt{2D}\mathrm{d}W_{s}, \quad t \leq s \leq T,\quad
 X_{t}  \sim \rho_{t}\left(x\right).
\end{aligned}
\end{equation}
Here $u_{s}\left(\cdot\,;q_{s}\right), \alpha_{s}\left(\cdot\,;q_{s}\right)$ indicate the possible   dependence on the fluid field $\lbrace q_s\rbrace_{t\leq s\leq T}$ for the transport velocity and the control. Molecular diffusion effect is represented by the Gaussian white noise with a diffusivity
coefficient $D$.
We seek to describe the probability measure, describing the law of individuals $X_s\sim\rho_s\left(x\right)$. The governing equation will satisfy the Kolmogorov forward  equation  for the SDE \eqref{eq:ctrl_sde}
\begin{equation}
\partial_{s}\rho_{s}\left(x\right)+\nabla\cdot\left[\left(u_{s}\left(x;q_{s}\right)+\alpha_{s}\left(x;q_{s}\right)\right)\rho_{s}\left(x\right)\right]=D\Delta \rho_{s}\left(x\right), \quad t \leq s \leq T,
\label{eq:KP}
\end{equation}
with the initial condition from the initial tracer configuration in the fluid field, $X_{t} \sim \rho_{t}\left(x\right)$. Notice that given a fluid vorticity state $q_s$, the law of individuals $\rho_s$ does not necessary coincide with $q_s$ (that is, the controlled tracer density field may not exactly track the evolution of the vorticity flow). 

In parallel to the Eulerian vorticity equation \eqref{eq:vort_eqn}, we can propose the external forcing in the following form acting the equivalent role of moving the tracers to the target density
\begin{equation}
f\left(x,s\right)=-\nabla\cdot\left[\alpha_{s}\left(x\right)q_{s}\left(x\right)\right],\label{eq:forcing}
\end{equation}
according to the same control effect $\alpha_{s}$ in the SDE formulation \eqref{eq:ctrl_sde}. The general damping term can be also simplified to the linear form $D\left(\Delta\right)=D\Delta$ accordingly.  The resulting flow equation becomes
\begin{equation}
\partial_{s}q_{s}\left(x\right)+\nabla\cdot\left[\left(u_{s}\left(x;q_{s}\right)+\alpha_{s}\left(x;q_{s}\right)\right)q_{s}\left(x\right)\right]=D\Delta q_{s}\left(x\right), \quad t \leq s \leq T,
\label{eq:ctrl_pde}
\end{equation}
exploiting the divergence free velocity field, $\nabla\cdot u_s=0$. 
Notice that the vorticity field $q_s$ is not passively advected by
the fluid field and actually determines the velocity field $u_s$.
Nevertheless, in the control problems discussed next, we can view the field
$q_s$ as \emph{a proper measure} on the state space since it
is solved by the continuity equation \eqref{eq:KP} up to a constant normalization constant.
In particular, if we set the initial condition $q_t=\rho_t$,  the flow vorticity field $q_s$ will track the tracer density measure in \eqref{eq:ctrl_sde} generated by the stochastic samples advected by the velocity field $u_{s}\left(x;q_{s}\right)$.
The equivalence between the controlled flow equation \eqref{eq:ctrl_pde} and the SDE for an ensemble of tracers implies the possibility to control the key flow structures by acting on the measurements of Lagrangian particles.


\section{Control of the fluid fields with mean field games}\label{sec:control_models}
Here, we propose the control models for the transport of both the tracer density and the fluid vorticity fields. We first formulate the problems based on the PDEs. Then, the corresponding stochastic control for the Lagrangian particles can be formulated in an equivalent way.

\subsection{Control of the tracer density field with a given fluid solution }
First, we propose the mean-field game model (MFG-1) concerning \textit{a mean-field game for indistinguished individuals with a given fluid dynamics.}
Given the flow field vorticity solution $q_s\left(x\right),t\leq s \leq T$, consider the optimal control problem about the value function on the optimal solutions $\rho\left(\cdot\right)\coloneqq\rho_{s}\left(x\right),\alpha\left(\cdot\right)\coloneqq\alpha_{s}\left(x\right), s\in\left[t,T\right], x\in\mathbb{T}^d$
\begin{equation}\label{eq:loss_tracer1}
\begin{aligned}
\mathcal{U}\left(\rho,t\right) & \coloneqq  \inf_{\rho(\cdot),\alpha(\cdot)}\mathcal{J}\left(\rho(\cdot),\alpha(\cdot)\right) =\inf_{\rho(\cdot),\alpha(\cdot)}  \bbs{\int G(x,q_T) \rho_{T} \left(x\right)\mathrm{d}x +\int_{t}^{T}\left\{ \int [L\left(\alpha_{s}\left(x\right )\right) +F\left(x,q_{s} \right)]\rho_{s}\left(x\right)\mathrm{d}x\right\} \mathrm{d}s }, \\
\st & \,\, \partial_{s}\rho_{s}  +\nabla\cdot\left[\left(\mathcal{T}q_{s}+\alpha_{s}\right)\rho_{s}\right]=D\Delta \rho_{s},\quad t\leq s\leq T,\quad \mathrm{and}\; \rho_{t}\left(x\right)=\rho\left(x\right).
\end{aligned}
\end{equation}
Here, the \emph{running cost} $L\left(\alpha\right)$ depends on the individual control action; the \emph{individual running cost} $F(x,q)$ represents the running cost interacting with the fluid field; and the \emph{terminal cost} $G(x,q)$ depends on the terminal flow state. The dependence on fluid field can be either globally or locally, and we will describe the specific forms of these functionals later.
Here, the solution $\rho_{s}$ of the continuity equation \eqref{eq:loss_tracer1} can be viewed as the density field of the controlled tracers \eqref{eq:ctrl_sde}, driven by the velocity field $u_s$ generated by the given flow vorticity solution $\mathcal{T}q_{s}$. We will refer the optimization problem \eqref{eq:loss_tracer1} as \textbf{MFG-1 model}.  The solution of MFG-1 model describes the optimal control of the tracer density from the initial configuration $\rho$ to the final target $\rho_{T}$.

The following proposition for the Euler-Lagrange equations describes the optimal solution of the MFG-1 model given a fixed flow vorticity solution $q_s$.
\begin{proposition}
Given the flow vorticity field $q_s\left(x\right),s\in \left[t,  T\right]$, the optimal tracer density $\rho_{s}\left(x\right)$ under the cost function \eqref{eq:loss_tracer1} with $L\left(\alpha\right)=\frac{1}{2}\vert\alpha\vert^2$ is given by the solution of the following \textbf{MFG-1 system}
\addtocounter{equation}{0}\begin{subequations}\label{eq:mfg1}
\begin{align}
\partial_{s}\rho_{s}+\nabla\cdot\left[\left(\mathcal{T}q_{s}+\nabla\phi_{s}\right)\rho_{s}\right] & =D\Delta \rho_{s}, \quad t\leq s\leq T, \label{eq:fpe1}\\
   \rho_{t}\left(x\right)&=\rho\left(x\right),\nonumber\\
\partial_{s}\phi_{s}+\frac{1}{2}\left|\nabla\phi_{s}\right|^{2}+\nabla\phi_{s}\cdot\mathcal{T}q_{s}+ D\Delta\phi_{s}  & = F(x,q_s), \quad t\leq s\leq T, \label{eq:hje1}\\
\phi_{T}\left(x\right)&=-G(x,q_T).\nonumber
\end{align}
\end{subequations}
The corresponding optimal control can be solved by $\alpha_{s}\left(x\right) = \nabla \phi_{s}\left(x\right)$.
\end{proposition}
\begin{proof}
We derive the Euler-Lagrangian equations for the MFG-1 model \eqref{eq:loss_tracer1} based on the proposed cost functionals.
Introduce the Lagrangian multiplier $\phi_{s}$
as 
\begin{align*}
\mathcal{E}\left(\phi_{s},\alpha_{s},\rho_{s},\rho_{T}\right)= & \int_{t}^{T} \int\left\{\phi_{s}\left(x\right)\left[\partial_{s}\rho_{s}\left(x\right)+\nabla\cdot\left(\left(\mathcal{T}q_{s}\left(x\right)+\alpha_{s}\left(x\right)\right)\rho_{s}\left(x\right)\right)-D\Delta \rho_{s}\left(x\right)\right]\right.\\
 & \left.+\frac{1}{2}\left|\alpha_{s}\left(x\right)\right|^{2}\rho_{s}\left(x\right)+   F\left(x,q_{s} \right) \rho_s\left(x\right)  \right\}\mathrm{d}x \mathrm{d}s+\int G\left(x,q_{T}\right)\rho_T\left(x\right)\mathrm{d}x.
\end{align*}
The Euler-Lagrangian equations can be found by taking
the variational derivatives on the input functions, so that
\[
\begin{aligned}\frac{\delta\mathcal{E}}{\delta\phi_{s}}= & \partial_{s}\rho_{s}+\nabla\cdot\left[\left(\mathcal{T}q_{s}+\alpha_{s}\right)\rho_{s}\right]-D\Delta \rho_{s}=0,\\
\frac{\delta\mathcal{E}}{\delta\alpha_{s}}= & \rho_{s}\alpha_{s}-\rho_{s}\nabla\phi_{s}=0,\\
\frac{\delta\mathcal{E}}{\delta \rho_{s}}= & \frac{1}{2}\left|\alpha_{s}\right|^{2}-\nabla\phi_{s}\cdot\left(\mathcal{T}q_{s}+\alpha_{s}\right)+ F\left(x,q_{s} \right)-\partial_{s}\phi_{s}-D\Delta\phi_{s}=0,\\
\frac{\delta\mathcal{E}}{\delta \rho_{T}}= & G\left(x,q_{T} \right)+\phi_{T}=0.
\end{aligned}
\]
The first equation gives the continuity equation for the controlled flow dynamics. The optimal control forcing $\alpha_s=\nabla \phi_s$ is discovered by the second equation provided $\rho_s>0$. The third and fourth equations yield the backward equation for the Lagrangian multiplier $\phi_s$ with its terminal condition. Thus we have a decoupled MFG system \eqref{eq:mfg1} for MFG-1 model.
\end{proof}

\subsection{Control of the vorticity flow field as a potential mean field game}
Next, we design the control of the vorticity flow stated in \eqref{eq:vort_eqn}
as a mean field game (MFG) system. In particular, we are seeking the
control forcing in the specific form as \eqref{eq:forcing}, 
where $\alpha_{s}$ is the additional control vector that aims
to drive the initial state to the final target. Combining with the
original vorticity equation \eqref{eq:vort_eqn}, we find the control
PDE associated with the control forcing \eqref{eq:forcing}
\begin{equation}
\partial_{s}q_{s}+\nabla\cdot\left[\left(u_{s}+\alpha_{s}\right)q_{s}\right]=D\Delta q_{s}, \quad t \leq s \leq T
\label{eq:continuity}
\end{equation}
with the prescribed initial state $q\mid_{s=t}=q\left(x\right)$.
Above, we define $u_{s}=\mathcal{T}q_{s}$ for the incompressible flow $\nabla\cdot u_{s}=0$.
The equation describes the control of the vorticity flow field $q_{s}$
from the initial state $q$ at $s=t$ to a final target
field $q_{T}$ at $s=T$ through the control $\alpha$.

Here, we consider the potential mean-field game model (MFG-2) concerning the control of the fluid vorticity field. Instead of controlling the tracer density field $\rho_{s}$ advected by the flow field, we consider the direct optimal control of the flow vorticity state
\begin{equation}\label{eq:loss_pde}
\begin{aligned}
\mathcal{V}\left(q,t\right) & \coloneqq\inf_{q\left(\cdot\right),\alpha\left(\cdot\right)}\mathcal{I}\left(q\left(\cdot\right),\alpha\left(\cdot\right)\right)=\inf_{q\left(\cdot\right),\alpha\left(\cdot\right)} \bbs{ \mathcal{G}\left(q_{T}\right)+\int_{t}^{T}\left[ \int L\left(\alpha_{s}\right)q_{s}\left(x\right)\mathrm{d}x +\mathcal{F}\left(q_{s}\right)\right]\mathrm{d}s },\\
\st & \,\, \partial_{s}q_{s}  +\nabla\cdot\left[\left(\mathcal{T}q_{s}+\alpha_{s}\right)q_{s}\right]=D\Delta q_{s},\quad t\leq s\leq T,\quad \mathrm{and}\;q_{t}\left(x\right)=q\left(x\right). 
\end{aligned}
\end{equation}
We will refer the optimization problem \eqref{eq:loss_pde} as \textbf{MFG-2 model}.
In the above potential game model, we assume that the running cost $\mathcal{F}$ and  the terminal cost $\mathcal{G}$ satisfy  
 $F(x,q)=\frac{\delta \mathcal{F}}{\delta q}\left(q,x\right), \, G(x,q)=\frac{\delta \mathcal{G}}{\delta q}\left(q,x\right)$, as a connection to the MFG-1 model \eqref{eq:loss_tracer1}. 
In contrast to the MFG-1 model \eqref{eq:loss_tracer1}, the flow vorticity $q_s$ becomes the controlled state in the MFG-2 model rather than the tracer density $\rho_{s}$. This leads to the nonlinear continuity equation in \eqref{eq:loss_pde}.

Still, we can solve the optimization problem by deriving the Euler-Lagrangian equations. The following proposition provides the Euler-Lagrangian equations for the optimal solution of the MFG-2 model.
\begin{proposition}
The optimally controlled flow solution in MFG-2 model \eqref{eq:loss_pde} with $L\left(\alpha\right)=\frac{1}{2}\vert\alpha\vert^2$ is given by $q_s$ and the associated optimal control for the flow field is given by $\alpha_{s}=\nabla \varphi_{s}$, where $\left(q_s,\varphi_{s}\right), s \in \left[t,T\right]$ solves the following \textbf{MFG-2 system}
\addtocounter{equation}{0}\begin{subequations}\label{eq:mfg2} 
\begin{align}
\partial_{s}q_{s}+\nabla\cdot\left[\left(\mathcal{T}q_{s}+\nabla\varphi_{s}\right)q_{s}\right] & =D\Delta q_{s}, \quad t\leq s\leq T,\label{eq:fpe2}\\
   q_{t}\left(x\right)&=q\left(x\right),\nonumber\\
\partial_{s}\varphi_{s}+\frac{1}{2}\left|\nabla\varphi_{s}\right|^{2}+\nabla\varphi_{s}\cdot\mathcal{T}q_{s}+\mathcal{T}^{*}\cdot\left(\nabla\varphi_{s}q_{s}\right) + D\Delta\varphi_{s} & =F(x,q_s), \quad t\leq s\leq T, \label{eq:hje2}\\
\varphi_{T}\left(x\right)&=-G(x,q_T).\nonumber
\end{align}
\end{subequations}
\end{proposition}
\begin{proof}
To derive the MFG-2 system for the problem \eqref{eq:loss_pde}
based on the proposed cost functionals,
 we introduce the Lagrangian multiplier $\varphi_{s}$
as
\begin{align*}
\mathcal{E}\left(\varphi_{s},\alpha_{s},q_{s},q_{T}\right)= & \int_{t}^{T}\left\{ \int\left[\varphi_{s}\left(x\right)\left[\partial_{s}q_{s}\left(x\right)+\nabla\cdot\left(\left(\mathcal{T}q_{s}\left(x\right)+\alpha_{s}\left(x\right)\right)q_{s}\left(x\right)\right)-D\Delta q_{s}\right]\right.\right.\\
 & \left.\left.+\frac{1}{2}\left|\alpha_{s}\left(x\right)\right|^{2}q_{s}\left(x\right)\right]\mathrm{d}x+\mathcal{F}\left(q_s\right)\right\} \mathrm{ds}+\mathcal{G}\left(q_{T}\right).
\end{align*}
The Euler-Lagrangian equations can be found by taking
the variational derivatives, so that
\[
\begin{aligned}\frac{\delta\mathcal{E}}{\delta\varphi_{s}}= & \partial_{s}q_{s}+\nabla\cdot\left[\left(\mathcal{T}q_{s}+\alpha_{s}\right)q_{s}\right]-D\Delta q_{s}=0,\\
\frac{\delta\mathcal{E}}{\delta\alpha_{s}}= & \alpha_{s}q_{s}-\nabla\varphi_{s}q_{s}=0,\\
\frac{\delta\mathcal{E}}{\delta q_{s}}= & \frac{1}{2}\left|\alpha_{s}\right|^{2}-\nabla\varphi_{s}\cdot\left(\mathcal{T}q_{s}+\alpha_{s}\right)-\partial_{s}\varphi_{s}-D\Delta\varphi_{s}-\mathcal{T}^{*}\cdot\left(\nabla\varphi_{s}q_{s}\right)+\frac{\delta\mathcal{F}}{\delta q_{s}}=0,\\
\frac{\delta\mathcal{E}}{\delta q_{T}}= & \frac{\delta\mathcal{G}}{\delta q_{T}}+\varphi_{T}=0.
\end{aligned}
\]
Define the Legendre transform $\sup_{\alpha}\left\{ p\cdot\alpha- L\left(\alpha\right)\right\} =\sup_{\alpha}\left\{ p\cdot\alpha-\frac{1}{2}\left|\alpha\right|^{2}\right\} =\frac{1}{2}\left|p\right|^{2}$.
The second equation above implies that the optimal solution is solved
by $\alpha_{s}\left(x\right)=\nabla\varphi_{s}\left(x\right)$.
\end{proof}

\subsection{SDE formulations for the mean field game models}
From a different perspective, we can reformulate the above PDE models for controlling tracer density and fluid vorticity as SDEs representing particles in Lagrangian form. The SDE formulation can also help to propose effective computational strategies based on simulation of the stochastic samples.

First, the continuity equation in the MFG-1 model \eqref{eq:loss_tracer1} can be viewed as the density equation of the motion of the stochastic particles \eqref{eq:ctrl_sde}. Therefore, we can formulate the following stochastic optimal control problem (referred as \textbf{MFG-1$'$ model}) for corresponding stochastic control problem with the given flow solution $q_s\left(\mathbf{x}\right),t\leq s \leq T$
\begin{equation}\label{eq:loss_sde_tracer1}
\begin{aligned}
\mathcal{U}\left(\rho,t\right) & \coloneqq \inf_{\rho\left(\cdot\right),\alpha\left(\cdot\right)}\mathcal{J}\left(\rho\left(\cdot\right),\alpha\left(\cdot\right)\right)=\inf_{\rho\left(\cdot\right),\alpha\left(\cdot\right)} \mathbb{E}\left\{ G(X_T,q_T)  +\int_{t}^{T}  [L\left(\alpha_{s}\left(X_s\right )\right) +F\left(X_{s},q_{s} \right)] \mathrm{d}s \right\}, \\
\st & \,\,\mathrm{d}X_{s}=\left[\mathcal{T}q_{s}\left(X_s\right)+\alpha_{s}\left(X_s\right)\right]\mathrm{d}s+\sqrt{2D}\mathrm{d}W_{s},\quad t\leq s\leq T,\quad \mathrm{and}\quad X_{t} \sim \rho\left(x\right).
\end{aligned}
\end{equation}
The optimal solution of the MFG-1$'$ model can be still solved by the Euler-Lagrangian equations.
\begin{corollary}
With the flow vorticity field $q_s,s\in \left[t,  T\right]$ given,  the solution $\left(\rho_s,\phi_s\right)$ of the MFG-1 system \eqref{eq:mfg1} provides the optimal solution of the stochastic MFG-1$'$ model \eqref{eq:loss_sde_tracer1}.
\end{corollary}

Then, the control forcing \eqref{eq:forcing} can be viewed as
the additional drifting effect $\alpha_{s}$ that is externally
exerted on the local particles, which in an optimal way is given by $\alpha_{s}=\nabla \varphi_s$. 
We can introduce the McKean-Vlasov equation
\[
\mathrm{d}X_{s}=\left[u_{s}\left(X_{s};\rho_{s}\right)+\alpha_{s}\left(X_{s};\rho_{s}\right)\right]\mathrm{d}s+\sqrt{2D}\mathrm{d}W_{s}.
\]
The tracer density as the law of the above stochastic particle $X_s\sim\rho_s$ is solved by the continuity equation  in the form
\[
\partial_{s}\rho_{s}+\nabla\cdot\left[\left(u_{s}+\alpha_{s}\right)\rho_{s}\right]=D\Delta \rho_{s}.
\]
Let $q_s$ be a solution of the controlled flow vorticity equation \eqref{eq:continuity}. With the  same form of $\alpha_{s}$ and $u_{s}=\mathcal{T}q_{s}$,  we have $\rho_s=q_s$ due to the uniqueness for linear equations. 

Similarly, according to the MFG-2 model \eqref{eq:loss_pde},
we give the following stochastic optimal control problem (referred as \textbf{MFG-2$'$ model})
\begin{equation}\label{eq:loss_pde1}
\begin{aligned}
\mathcal{V}\left(q,t\right) & \coloneqq \inf_{q\left(\cdot\right),\alpha\left(\cdot\right)}\mathcal{I}\left(q\left(\cdot\right),\alpha\left(\cdot\right)\right)=\inf_{q\left(\cdot\right),\alpha\left(\cdot\right)} \mathbb{E}\left\{ \tilde{G}(X_T,q_T)  +\int_{t}^{T}  [L\left(\alpha_{s}\left(X_s\right )\right) +\tilde{F}\left(X_{s},q_{s} \right)] \mathrm{d}s \right\}, \\
\st & \,\,\mathrm{d}X_{s}=\left[\mathcal{T}q_{s}\left(X_{s}\right)+\alpha_{s}\left(X_{s}\right)\right]\mathrm{d}s+\sqrt{2D}\mathrm{d}W_{s},\quad t\leq s\leq T,\quad \mathrm{and}\quad X_{s} \sim q\left(x\right).
\end{aligned}
\end{equation}
Above, $q_s$ as the law of the random particles $X_s$ is playing the equivalent role of the tracer density in the nonlinear field, and the initial sample is drawn from the initial flow state $X_{t} \sim q$. Still, we assume the MFG-2 model \eqref{eq:loss_pde} is a potential game with the cost functionals $\mathcal{F},\mathcal{G}$ satisfying 
\[
\mathcal{F}\left(q\right) = \int \tilde{F}\left(x,q\right)q(x)\mathrm{d}x, \quad \mathcal{G}\left(q\right) = \int \tilde{G}\left(x,q\right)q(x)\mathrm{d}x.
\]
In a similar way, the solution of the MFG-2$'$ model is given by the same Euler-Lagrangian equations in \eqref{eq:mfg2}. Notice that the new costs $\tilde{F},\tilde{G}$ may not necessarily be the same as the $F,G$ in \eqref{eq:loss_sde_tracer1}.
\begin{corollary}
With the same initial condition $q=\rho$,  the solution $\left(\rho_s,\varphi_s\right)$ of the MFG equations \eqref{eq:mfg2} provides the optimal solution of the stochastic mean field game model \eqref{eq:loss_pde1}.
\end{corollary}

\subsubsection*{Mean field game with finite number of players}
 The empirical distribution of a group of
$N$ particles is supposed to recover the global vorticity field
\begin{equation}
q_{s}^{N}\left(x\right)=\frac{1}{N}\sum_{i=1}^{N}\delta_{X_{s}^{i}}\left(x\right) \cong q_{s}\left(x\right),\label{eq:emp_pdf}
\end{equation}
as $N\rightarrow\infty$. If the control velocity field $\vec{\alpha}$  for SDE \eqref{eq:ctrl_sde} and for   \eqref{eq:continuity} coincide,  the continuity PDE \eqref{eq:continuity}
becomes the   Kolmogorov forward  equation  for the SDE \eqref{eq:ctrl_sde}
for particles, thus $q_{s}$ plays the equivalent role of a probability
density measure that uniformly describes the statistics of all the
particle trajectories $\left\{ X_{s}^{i}\right\} $. Notice
that the control $\alpha_{s}$ depends on the entire density
$q_{s}$, thus finding the optimal control finally leads to a MFG problem.

We can find the associated stochastic differential game with mean-field interaction based on the SDE model
\eqref{eq:ctrl_sde} using the proposed cost functionals in \eqref{eq:loss_pde}
\begin{equation}\label{eq:loss_sde}
\begin{aligned}
\mathcal{V}^{N}\left(q,t\right) & \coloneqq \inf_{q^{N}\left(\cdot\right),\alpha^{N}\left(\cdot\right)}\mathcal{I}_{i}\left(q^{N}\left(\cdot\right),\alpha^{N}\left(\cdot\right)\right)= \inf_{q^{N}\left(\cdot\right),\alpha^{N}\left(\cdot\right)}\mathbb{E}\left\{ \tilde{G}\left(X_T^{i},q_{T}^{N}\right)+\int_{t}^{T}\left[L\left(\alpha_{s}^{N}\left(X_{s}^{i}\right)\right)+\tilde{F}\left(X_s^{i},q_s^{N}\right)\right]\mathrm{d}s\right\} ,\\
\st & \,\,\mathrm{d}X_{s}^{i}  =\left[\mathcal{T}q_{s}^{N}\left(X_{s}^{i}\right)+\alpha_{s}^{N}\left(X_{s}^{i}\right)\right]\mathrm{d}s+\sqrt{2D}\mathrm{d}W_{s}^{i},\quad t\leq s\leq T,\;X_{t}^{i}\sim q,\quad i=1,\cdots,N.
\end{aligned}
\end{equation}
Above, $q_{s}^{N}$ is the empirical distribution \eqref{eq:emp_pdf} of the group of finite players $\lbrace X_s^{i}\rbrace_{i=1}^{N}$, and $\alpha_{s}^{N}$ is implicitly dependent on the empirical distribution of the $N$ samples.
The advantage of using this finite particle model \eqref{eq:loss_sde} is that we are able to find the optimal control for the complex flow field by controlling the finite number of Lagrangian tracers. This enables an efficient Monte-Carlo type approach to capture the continuous fluid solution. We refer to \cite{lasry2007mean} for the mean-field limit from the above stochastic differential game with mean-field interaction to the mean-field game system in the central planer form.

\section{Functional Hamilton-Jacobi equations for the value functions}\label{sec:hje}

In this section, we first demonstrate that the MFG models in \eqref{eq:mfg1} and \eqref{eq:mfg2} can be recast in 
Hamiltonian forms. Thus the value functions $\mathcal{U}\left(\rho,t\right)$
and $\mathcal{V}\left(\rho,t\right)$ satisfy the corresponding functional Hamilton-Jacobi
equations. Under the functional HJE formulations, the coupled MFG-2 model can be shown to be related to a modified MFG-1 model with given optimal solutions.

\subsection{Functional HJE for the MFG-1 model}

First, the MFG-1 model \eqref{eq:loss_tracer1} for controlling tracer
density transport has the Hamiltonian functional
\begin{equation}
\mathcal{H}_{1}\left(\rho,\phi;q\right)=\int\left[\frac{1}{2}\left|\nabla\phi\right|^{2}+\nabla\phi\cdot\mathcal{T}q\left(x\right)+D\Delta\phi-F\left(x,q\right)\right]\rho\left(x\right)\mathrm{d}x,\label{eq:hamil_mfg1}
\end{equation}
with $q\left(x\right)$ the solution of the prescribed flow vorticity
field. It can be found that the Euler-Lagrangian equations \eqref{eq:mfg1} follow the Hamiltonian dynamics
for $s\in\left[t,T\right]$
\begin{equation}
\begin{aligned}\partial_{s}\rho_{s} & =\frac{\delta\mathcal{H}_{1}}{\delta\phi}\left(\rho_{s},\phi_{s};q_{s}\right),\\
\partial_{s}\phi_{s} & =-\frac{\delta\mathcal{H}_{1}}{\delta\rho}\left(\rho_{s},\phi_{s};q_{s}\right),
\end{aligned}
\label{eq:hamil1}
\end{equation}
with $\rho_{t}=\rho$, and $q_{s}$ given by the background advection flow solution.
In particular, since $q\left(x\right)$ is given we can define the local Hamiltonian function $H_1$ according to
the separable Hamiltonian functional \eqref{eq:hamil_mfg1} as
\begin{equation}
H_1\left(x,p; q\right)=\frac{1}{2}\left|p\right|^{2}+\mathcal{T}q\left(x\right)\cdot p,\label{eq:hamil1_local}
\end{equation}
and the corresponding Lagrangian becomes
\[
L\left(x,b; q\right)=\sup_{p}\lbrace b\cdot p-H_1\left(x,p;q\right)\rbrace=\frac{1}{2}\left|b-\mathcal{T}q\left(x\right)\right|^{2},
\]
and with the convexity the supremum is reached at $p^{*}=b-\mathcal{T}q\left(x\right)$.  Thus, the Lagrangian functional can be redefined as $L\left(x,\alpha; q\right) \coloneqq L\left(x,b\left(x,\alpha\right); q\right)=\frac{1}{2}\left|\alpha\right|^{2}$.
The Hamiltonian functional \eqref{eq:hamil_mfg1} can be represented by the local Hamiltonian \eqref{eq:hamil1_local} as
\[
\mathcal{H}_{1}\left(\rho,\phi;q\right)=\int \left[H_1\left(x,\nabla\phi;q\right)+D\Delta\phi-F\left(x,q\right)\right]\rho\left(x\right)\mathrm{d}x.
\]
And we can introduce
the optimal `total velocity' as $b_{s}\left(x\right)=\partial_{p}H_1\left(x,\nabla\phi_{s};q_{s}\right)=\mathcal{T}q_{s}\left(x\right)+\nabla\phi_{s}$.

Since the MFG-1 system  \eqref{eq:mfg1} is decoupled, we can directly compute the value function and it's relation with the HJE solution $\phi_s$, $t\leq s\leq T$. Let $\mathcal{U}\left(\rho,t;q\right)$ be the optimal value function
from the MFG-1 model \eqref{eq:loss_tracer1} with the given flow solution $q_{t}=q$. The following proposition gives the functional HJE for the MFG-1 model.
\begin{proposition}\label{prop:hje1_funcnal}
Given a fluid vorticity field $\lbrace q_s\rbrace_{ t\leq s\leq T},$
let $\mathcal{U}\left(\rho,t;q(\cdot)\right)$ be  the   optimal value function of the MFG-1 model \eqref{eq:loss_tracer1}, whose minimizer is the classical solution $(\rho_s, \phi_s)$ to  the MFG system \eqref{eq:mfg1}. Then we have
$\phi_t = - \frac{\delta\mathcal{U}}{\delta\rho}\left(\rho_{t},x,t;q\left(\cdot\right)\right)$ for any $t\leq T$ and the value function satisfies the functional HJE
\begin{equation}
\begin{aligned}
\partial_{t}\mathcal{U}\left(\rho,t;q(\cdot)\right) &-\mathcal{H}_{1}\left(\rho,-\frac{\delta\mathcal{U}}{\delta\rho}\left(\rho,x,t;q(\cdot)\right);q_t\right) =0, \quad \forall t\leq T,\\
& \mathcal{U}\left(\rho,T;q(\cdot)\right) =\int G(x,q_T) \rho\left(x\right)\mathrm{d}x.\label{eq:hje_funal1}
\end{aligned}
\end{equation}
\end{proposition}
In the above, given a fluid vorticity field $q\left(\cdot\right)\coloneqq \lbrace q_s\rbrace_{ t\leq s\leq T}$, we denote  $\mathcal{U}\left(\rho,t;q(\cdot)\right)\coloneqq\mathcal{U}\left(\rho,t;\lbrace q_s\rbrace_{t\leq s\leq T}\right)$ to indicate that the dependence of $\mathcal{U}$ on $q\left(\cdot\right)$ is in terms of  the whole given absolute continuous curve, while the dependence of $\mathcal{U}$ on $\rho$ is only in terms of the initial state $\rho_t=\rho.$
\begin{proof}
Given a fluid vorticity field $\lbrace q_s\rbrace_{s\leq T}$, recall  \eqref{eq:mfg1}.
Multiplying \eqref{eq:hje1} by $\rho_s$ and integrating time from $t$ to $T$, we have
\begin{equation}
\begin{aligned}
&\int_t^T \int\rho_s \left(\partial_{s}\phi_{s}+\frac{1}{2}\left|\nabla\phi_{s}\right|^{2}+\nabla\phi_{s}\cdot\mathcal{T}q_{s}+ D\Delta\phi_{s} - F(x,q_s) \right) \ud x\ud s\\
=&\int_t^T \int\left( \pt_s (\rho_s \phi_{s}) + \big[-\pt_s \rho_s   - \nabla \cdot (\rho_s \mathcal{T}q_{s}) + D\Delta \rho_s \big] \phi_{s} +\frac{1}{2}\left|\nabla\phi_{s}\right|^{2} \rho_s   -  F(x,q_s) \rho_s \right) \ud x\ud s\\
=&  -\int\rho_T G(x,q_T)\ud x- \int\rho_t \phi_{t}\ud x + \int_t^T \int\left(-\frac{1}{2}\left|\nabla\phi_{s}\right|^{2} \rho_s   -  F(x,q_s) \rho_s \right) \ud x\ud s,
\end{aligned}
\end{equation}
where in the last equality we used \eqref{eq:fpe1} for the dynamics of $\rho_s$ and the terminal condition for $\phi_s$.

Since $\rho_s$ and $\phi_s$ for $t\leq s\leq T$ are optimal curves solving the Euler-Lagrangian \eqref{eq:mfg1} with $\rho_t = \rho$, the corresponding value function is given by
\begin{equation}\label{eq:value1}
\mathcal{U}(\rho, t;q\left(\cdot\right)) = \int \left[\rho_T G(x,q_T) + \int_t^T \left(\frac{1}{2}\left|\nabla\phi_{s}\right|^{2} \rho_s   +  F(x,q_s) \rho_s \right) \ud s \right] \ud x = - \int \rho \phi_{t} \ud x
\end{equation}
for any $t\leq T$. We can directly verify that $\phi_t = - \frac{\delta\mathcal{U}}{\delta\rho}\left(\rho_{t},x,t;q\right)$ for any $t\leq T$ from the above identity.
Furthermore,  this leads to the following dynamical equation for the value function $\mathcal{U}$ by taking time derivation on both sides of \eqref{eq:value1} and applying the equation \eqref{eq:hje1}
 \[
 \partial_t \mathcal{U}=-\int\rho\partial_t\phi_t\ud x=\int\left[H_1\left(x,\nabla\phi_t;q_t\right)+D\Delta\phi_t-F\left(x,q_t\right)\right]\rho\left(x\right)\mathrm{d}x.
 \]
From the above calculations and the definition of the funtional Hamiltonian $\mathcal{H}_1$ \eqref{eq:hamil_mfg1}, it confirms that the value function satisfies the functional HJE  \eqref{eq:hje_funal1}.
\end{proof}

\subsection{Functional HJE for the MFG-2 model}

Similarly, the Hamiltonian functional for the MFG-2 model \eqref{eq:loss_pde}
can be found as
\begin{equation}
\mathcal{H}_{2}\left(q,\varphi\right)=\int \left[\left(\frac{1}{2}\left|\nabla\varphi\right|^{2}+\nabla\varphi\cdot\mathcal{T}q\right)q\left(x\right)-D\nabla\varphi\cdot\nabla q\right]\mathrm{d}x-\mathcal{F}\left(q\right).\label{eq:hamil_mfg2}
\end{equation}
The following Hamiltonian equations are still valid for the MFG-2 model \eqref{eq:mfg2}
\begin{equation}
\begin{aligned}\partial_{s}q_{s} & =\frac{\delta\mathcal{H}_{2}}{\delta\varphi}\left(q_{s},\varphi_{s}\right),\\
\partial_{s}\varphi_{s} & =-\frac{\delta\mathcal{H}_{2}}{\delta q}\left(q_{s},\varphi_{s}\right).\label{eq:hamil2}
\end{aligned}
\end{equation}
We can also use the local Hamiltonian function to express the functional Hamiltonian
\begin{equation}
\begin{aligned}
H_2\left(x,p,q\right) &=\frac{1}{2}\left|p\right|^{2}+\mathcal{T}q\cdot p,\\
\mathcal{H}_{2}\left(q,\varphi\right) &=\int \left[H_2\left(x,\nabla\varphi(x), q\right)q\left(x\right)-D\nabla\varphi\cdot\nabla q\right]\mathrm{d}x-\mathcal{F}\left(q\right).\label{eq:hamil_func2}
\end{aligned}
\end{equation}
Above, we have $H_2\left(x,p,q\right)=\sup_{\alpha}\left\{ b\left(q,\alpha\right)\cdot p-L\left(\alpha\right)\right\} =\sup_{\alpha}\{ \left(\mathcal{T}q+\alpha\right)\cdot p-\frac{1}{2}\left|\alpha\right|^{2}\} $ with the optimal $\alpha^*=p$.
The optimal `total velocity' is defined in the same way as $b_{s}\left(x\right)=\partial_{p}H_2\left(q_s\left(x\right),\nabla\varphi_s\left(x\right)\right)=\mathcal{T}q_s\left(x\right)+\nabla\varphi_s\left(x\right)$. Notice that compared with \eqref{eq:hamil1} the Hamiltonian functional \eqref{eq:hamil_func2} becomes  non-separable, i.e., $H_2(x,p,q)$ is not in the form of $f_1(x,p)+f_2(x,q)$ \cite{gangbo2022mean}. This non-separable Hamiltonian comes from the nature of the original nonlinear fluid drift.

Therefore, we have the following proposition describing the master
equation of the MFG-2 model.
 
\begin{proposition}\label{prop:hje2_funcnal}
Assume the MFG-2 model \eqref{eq:loss_pde} achieves an optimal solution $(\rho_s,\phi_s)_{t\leq s\leq T}$, which solves MFG-2 system \eqref{eq:mfg2}.  Then the optimal value functions $\mathcal{V}\left(q,t\right)$ in the MFG-2 model \eqref{eq:loss_pde}  
is the solution to the following functional HJE 
\begin{equation}
\partial_{t}\mathcal{V}\left(q,t\right)-\mathcal{H}_2\left(q,-\frac{\delta\mathcal{V}}{\delta q}\right)=0, \,\,\, \forall t\leq T; \quad\mathcal{V}\left(q,T\right)=\mathcal{G}\left(q\right).\label{eq:master-mfg2}
\end{equation}
In detail, for the MFG-2 model we have
\begin{align*}
\partial_{t}\mathcal{V}\left(q,t\right)&-\frac{1}{2}\int\left|\nabla\frac{\delta\mathcal{V}}{\delta q}\left(q,x,t\right)\right|^{2}q\left(x\right)\mathrm{d}x+\int\left(\nabla\frac{\delta\mathcal{V}}{\delta q}\left(q,x,t\right)\cdot\mathcal{T}q\left(x\right)\right)q\left(x\right)\mathrm{d}x\\
& -D\int\nabla\frac{\delta\mathcal{V}}{\delta q}\left(q,x,t\right)\cdot\nabla q\left(x\right)\mathrm{d}x+\mathcal{F}\left(q\right)  =0,\quad t\leq T.
\end{align*}
In addition, the control solution $\partial_{s}\varphi_{s} =-\frac{\delta\mathcal{H}_{2}}{\delta q}\left(q_{s},\varphi_{s}\right)$ satisfies $\varphi_{s}\left(x\right)=-\frac{\delta\mathcal{V}}{\delta q}\left(q_{s},x,s\right)$.
\end{proposition}
The proof  is essentially in parallel with the proofs in \cite{gao2022master,cardaliaguet2019master} according to the above specific
Hamiltonian formulation \eqref{eq:hamil2}.
\begin{proof}
The property of the Lagrangian functional guarantees that for any
function $p_{s}$, we have
\begin{equation}
 \int L\left(\alpha_{s}\right)q_{s}\left(x\right)\mathrm{d}x
\geq  \int\left[\left(\mathcal{T}q_{s}\left(x\right)+\alpha_{s}\right)\cdot p_{s}\left(x\right)-H_{2}\left(x,p_{s},q_{s}\right)\right]q_{s}\left(x\right)\mathrm{d}x.\label{eq:lagrangian}
\end{equation}
Above, we use $b_{s}=\mathcal{T}q_{s}+\alpha_{s}$. The inequality is valid for any $p_{s}$ and the equality holds when $p_s=\alpha_s$.

By taking $p_{s}\left(x\right)=-\nabla\frac{\delta\mathcal{V}}{\delta q}\left(q_{s},x,s\right)$,
the right hand side of the above inequality \eqref{eq:lagrangian} yields
\begin{align*}
 & \int\left\{-\left[\mathcal{T}q_{s}\left(x\right)+\alpha_{s}\right]\cdot\nabla\frac{\delta\mathcal{V}}{\delta q}\left(q_{s},x,s\right)-H_{2}\left(x,-\nabla\frac{\delta\mathcal{V}}{\delta q}\left(q_{s},x,s\right),q_{s}\right)\right\}q_{s}\left(x\right)\mathrm{d}x +\mathcal{F}\left(q_s\right)\\
= & \int\left\{ \nabla\cdot\left[\left(\mathcal{T}q_{s}+\alpha_{s}\right)q_{s}\right]\frac{\delta\mathcal{V}}{\delta q}\left(q_{s},x,s\right)-H_{2}\left(x,-\nabla\frac{\delta\mathcal{V}}{\delta q}\left(q_{s},x,s\right),q_{s}\right)q_{s}\left(x\right)\right\} \mathrm{d}x+\mathcal{F}\left(q_s\right)\\
= & -\partial_{s}\mathcal{V}\left(q_{s},s\right)-\int\partial_{s}q_{s}\frac{\delta\mathcal{V}}{\delta q}\left(q_{s},x,s\right)\mathrm{d}x +D\int q_{s}\Delta\frac{\delta\mathcal{V}}{\delta q}\left(q_{s},x,s\right)\mathrm{d}x\\
 & +\partial_{s}\mathcal{V}\left(q_{s},s\right)  -\int H_{2}\left(x,-\nabla\frac{\delta\mathcal{V}}{\delta q}\left(q_{s},x,s\right),q_{s}\right)q_{s}\left(x\right)\mathrm{d}x+\mathcal{F}\left(q_s\right)\\
= & -\frac{\mathrm{d}}{\mathrm{d}s}\mathcal{V}\left(q_{s},s\right)+\partial_{s}\mathcal{V}\left(q_{s},s\right)-\mathcal{H}_{2}\left(q_{s},-\frac{\delta\mathcal{V}}{\delta q}\left(q_{s},x,s\right)\right).
\end{align*}
The second equality above uses the continuity equation \eqref{eq:fpe2} for $q_{s}$.
The third equality uses the total differentiation of $\mathcal{V}\left(q_{s},s\right)$ about $s$
and the definition of the Hamiltonian $\mathcal{H}_{2}$ with $\varphi_{s}\left(x\right)=-\frac{\delta\mathcal{V}}{\delta q}\left(q_{s},x,s\right)$.
Next, by taking the time integration between $t$ and $T$ and noticing
$\mathcal{V}\left(q_{T},T\right)=\mathcal{G}\left(q_{T}\right)$
and $\mathcal{V}\left(q_{t},t\right)=\mathcal{V}\left(q,t\right)$,
the left hand side of the inequality \eqref{eq:lagrangian} gives the
total cost $\mathcal{I}\left(q\left(\cdot\right),\alpha\left(\cdot\right)\right)$ in
\eqref{eq:loss_pde} with any solutions $\left(\rho_{s},\alpha_{s}\right)$
satisfying the continuity equation. Therefore, it confirms that if the
functional HJE \eqref{eq:master-mfg2} is satisfied, the solution always
gives the infinimum of the total cost.
\[
\underset{q\left(\cdot\right),\alpha\left(\cdot\right)}{\inf}\mathcal{I}\left(q\left(\cdot\right),\alpha\left(\cdot\right)\right)\geq\mathcal{V}\left(q,t\right).
\]
Finally, it is noticed that the optimal solution $q_s,\alpha_s=\nabla\phi_s$ of the optimization problem \eqref{eq:loss_pde}  is reached at the classical solution to the Euler-Lagrangian equations \eqref{eq:hamil2}. 
We see that the equality is achieved with $p_{s}\left(x\right)=\alpha_s\left(x\right)=\nabla\varphi_{s}\left(x\right)=-\nabla\frac{\delta\mathcal{V}}{\delta q}\left(q_{s},x,s\right)$. Thus the optimal value function $\mathcal{V}\left(q,t\right)$ is reached at the critical solution given by the Euler-Lagrangian equations.
This finishes the proof that $\mathcal{V}\left(q,t\right)$ gives
the minimum of the total cost.
\end{proof}

\subsection{A link between the MFG-1 and MFG-2 models}
From the previous discussions, we find that the MFG-1 model \eqref{eq:loss_tracer1} gives decoupled equations for $\rho_s$ and $\phi_s$, specially with linear continuity equation \eqref{eq:fpe1} given the flow and control fields $q_s$ and $\phi_s$.  On the other hand, the MFG-2 model \eqref{eq:loss_pde} leads to closely coupled forward and backward nonlinear equations, which require additional efforts for strategies to effectively solve the optimal functions.
Here, we establish a link between the two models, such that the MFG-2 model can be approached by an `approximated' form of the MFG-1 model. This enables the effective computational strategy that will be discussed next in Section~\ref{sec:comp_strategy}.

In order to develop separable approximate model, we introduce the functional Hamiltonian on the functions $\left(\tilde{q},\tilde{\phi}\right)$ based on some given solutions $\left(q,\varphi\right)$
\begin{equation}
\tilde{\mathcal{H}}_{2}\left(\tilde{q},\tilde{\phi};q,\varphi\right)=\int \left[\left(\frac{1}{2}\left|\nabla\tilde{\phi}\right|^{2}+\mathcal{T}q\cdot\nabla\tilde{\phi}+D\Delta\tilde{\phi}\right)\tilde{q}+\left(q\nabla\varphi\cdot\mathcal{T}\tilde{q}-F\left(x,q\right)\tilde{q}\right)\right]\mathrm{d}x.\label{eq:hamil_mfg2_mod}
\end{equation}
Comparing with the Hamiltonian \eqref{eq:hamil_mfg2}, we modify the coupling term, $\left(\nabla\varphi\cdot\mathcal{T}q\right)q$, into two separable components, $\left(\nabla\varphi\cdot\mathcal{T}\tilde{q}\right)q$ and $\left(\nabla\tilde{\phi}\cdot\mathcal{T}q\right)\tilde{q}$. The running cost $\mathcal{F}\left(q\right)$ is also replaced by the separable one, $\int F\left(x,q\right)\tilde{q}$. Therefore, this new Hamiltonian functional $\tilde{\mathcal{H}}_2$ fits into the MFG-1 model \eqref{eq:hamil_mfg1} with a modified running cost function
\[
\tilde{F}\left(x,q,\varphi\right)=F\left(x,q\right)-\mathcal{T}^{*}\cdot\left(q\nabla\varphi\right).
\]
The new Hamiltonian functional \eqref{eq:hamil_mfg2_mod} leads to the following decoupled forward and backward equations 
\addtocounter{equation}{0}\begin{subequations}\label{eq:mfg2_decouple} 
\begin{align}
\partial_{s}\tilde{q}_{s} & =\frac{\delta\tilde{\mathcal{H}}_{2}}{\delta\tilde{\phi}}\left(\tilde{q}_{s},\tilde{\phi}_{s};q_s,\varphi_s\right) = -\nabla\cdot\left[\left(\mathcal{T}q_{s}+\nabla\tilde{\phi}_{s}\right)\tilde{q}_{s}\right]  + D\Delta \tilde{q}_{s},\label{eq:fpe2_de}\\
\partial_{s}\tilde{\phi}_{s} & =-\frac{\delta\tilde{\mathcal{H}}_{2}}{\delta \tilde{q}}\left(\tilde{q}_{s},\tilde{\phi}_{s};q_s,\varphi_s\right) = -\frac{1}{2}\left|\nabla\tilde{\phi}_{s}\right|^{2}-\mathcal{T}q_{s}\cdot\nabla\tilde{\phi}_{s} - D\Delta\tilde{\phi}_{s} + \tilde{F}(x,q_s,\varphi_s), \label{eq:hje2_de}
\end{align}
\end{subequations}
with initial condition $\tilde{q}_{t}\left(x\right)=\tilde{q}\left(x\right)$ and final condition $\tilde{\phi}_{T}\left(x\right)=-G(x,q_T)$.
Notice that the  equations \eqref{eq:mfg2_decouple} agree with the MFG-1 model \eqref{eq:mfg1} using the new running cost function $\tilde{F}$. Therefore, the solutions $\left(\tilde{q}_{s},\tilde{\phi}_{s}\right),t\leq s\leq T$ to \eqref{eq:mfg2_decouple} solve the corresponding optimization problem as a \emph{modified MFG-1 model}
\begin{equation}\label{eq:loss_modified}
\begin{aligned}
\inf_{\tilde{q}\left(\cdot\right),\tilde{\alpha}\left(\cdot\right)} \bbs{\int G(x,q_T) \tilde{q}_{T} \left(x\right)\mathrm{d}x +\int_{t}^{T}\left\{ \int \left[L\left(\tilde{\alpha}_{s}\right) +  \tilde{F}\left(x,q_{s},\alpha_s \right)\right]\tilde{q}_{s}\left(x\right)\mathrm{d}x\right\} \mathrm{d}s }, \\
\st  \,\, \partial_{s}\tilde{q}_{s}  +\nabla\cdot\left[\left(\mathcal{T}q_{s}+\tilde{\alpha}_{s}\right)\tilde{q}_{s}\right]=D\Delta \tilde{q}_{s},\quad t\leq s\leq T,\quad \mathrm{and}\; \tilde{q}_{t}=\tilde{q},
\end{aligned}
\end{equation}
conditional on the given functions $\left(q_s,\varphi_s\right)$ with $\alpha_s=\nabla\varphi_s$. Here, we simply use $\tilde{F}\left(x,q,\varphi\right)$ to represent $\tilde{F}\left(x,q,\varphi\right)$ when $\alpha=\nabla \varphi$.  More important, if we can find the optimal solution such that, $\tilde{q}_s=q_s,\tilde{\alpha}_s=\alpha_s$, the optimal solution of \eqref{eq:mfg2_decouple} gives the solution to the MFG-2 model \eqref{eq:mfg2}. This implies that we may seek the optimal solution for the coupled MFG-2 model \eqref{eq:mfg2} through the solutions from the above decoupled MFG-1 model \eqref{eq:loss_modified} as a fixed point problem.

Next, if we take the function $\varphi_s$ from the solutions of the MFG-2 model \eqref{eq:hje2} according to any solution of $q_s$, this gives the optimal value function based on \eqref{eq:loss_modified}
\begin{equation}
\tilde{\mathcal{U}}\left(\tilde{q},t; q\left(\cdot\right)\right) = \inf_{\tilde{q}\left(\cdot\right),\tilde{\alpha}\left(\cdot\right)} \mathcal{J}\left(\tilde{q}\left(\cdot\right),\tilde{\alpha}\left(\cdot\right);q\left(\cdot\right)\right), \quad \forall t\leq T,\label{eq:optim_loss}
\end{equation}
where $\tilde{q}_t=\tilde{q}$ is the initial distribution of the tracers, and $q\left(\cdot\right)\coloneqq q_s\left(x\right),t\leq s\leq T$ indicates the dependence on the whole continuous curve. The following lemma can be found by comparing the  solutions in the corresponding HJEs.
\begin{lemma}\label{lem:ctrl_equiv}
Given any $t\leq T$ and $\left(q_s,\varphi_s\right),t\leq s\leq T$ to be the unique classical solutions of the MFG-2 model \eqref{eq:mfg2}, the optimal solution $\left(\tilde{q}_s,\tilde{\phi}_s\right)$ of \eqref{eq:loss_modified} with the optimal value function $\mathcal{U}\left(\tilde{q},t;q\left(\cdot\right)\right)$ exists uniquely and   satisfies the following relation
\begin{equation}\label{eq:phi_iden}
\tilde{\phi}_s\left(x\right) = -\frac{\delta\tilde{\mathcal{U}}}{\delta \tilde{q}}\left(\tilde{q}_s,x,s;q\left(\cdot\right)\right) = \varphi_s\left(x\right).
\end{equation}
\end{lemma}
\begin{proof}
Since $\varphi_s$ is given by the optimal solution of the MFG-2 model, it satisfies the equation from the Hamiltonian \eqref{eq:hamil_func2}
\[
\partial_{s}\varphi_{s} =-\frac{\delta\mathcal{H}_{2}}{\delta q}\left(q_{s},\varphi_{s}\right) =  -H_2\left(x, \nabla\varphi_{s}, q_{s}\right)-D\Delta\varphi_{s} + F\left(x,q_{s}\right) -\mathcal{T}^{*}\cdot\left(q_{s}\nabla\varphi_{s}\right).
\]
Similarly, $\tilde{\phi}_{s}$ is the optimal solution of the problem \eqref{eq:loss_modified}, which satisfies the MFG-1 model with the new running cost $\tilde{F}$. Thus using the Hamiltonian \eqref{eq:hamil_mfg2_mod} we have,
\[
\partial_{s}\tilde{\phi}_{s} =-\frac{\delta\tilde{\mathcal{H}}_{2}}{\delta \tilde{q}}\left(\tilde{q}_{s},\tilde{\phi}_{s};q_s,\varphi_s\right) =  -H_1\left(x,\nabla\tilde{\phi}_{s};q_s\right)-D\Delta\tilde{\phi}_{s} + \tilde{F}\left(x,q_{s},\varphi_s\right).
\]
Comparing the above two equations and noticing by definition that $H_1\left(x,p;q\right)=H_2\left(x,p,q\right)$ and $\tilde{F}\left(x,q_{s},\varphi_s\right)=F\left(x,q_{s}\right) -\mathcal{T}^{*}\cdot\left(q_{s}\nabla\varphi_{s}\right)$. In addition, $\varphi_T\left(x\right)=\tilde{\phi}_T\left(x\right)=-G\left(x,q_T\right)$ have the same terminal condition. 
We obtain $\tilde{\phi}_s=\varphi_s$ due to the uniqueness of solution to the HJE. And the first equality in \eqref{eq:phi_iden} comes by Proposition~\ref{prop:hje1_funcnal}.
\end{proof}

With the above lemma, we can link the two sets of optimal solutions $\left(\tilde{q}_s,\tilde{\phi}_s\right)$ and $\left(q_s,\varphi_s\right)$ from the MFG-1 and MFG-2 model accordingly as follows
\begin{theorem}\label{thm:link}
The optimal value function $\tilde{\mathcal{U}}\left(\tilde{q},t\right)$ in \eqref{eq:optim_loss} of the problem \eqref{eq:loss_modified} given the solutions $\left(q_{s},\varphi_{s}\right), t\leq s\leq T$ of the MFG-2 model \eqref{eq:mfg2} satisfy the following functional HJE
\begin{equation}
\begin{aligned}
\partial_{t}\tilde{\mathcal{U}}\left(\tilde{q},t\right) &-\frac{1}{2}\int\left|\nabla\frac{\delta\tilde{\mathcal{U}}}{\delta \tilde{q}}\left(\tilde{q},x,t\right)\right|^{2}\tilde{q}\left(x\right)\mathrm{d}x +\int F\left(x,q_t\right)\tilde{q}\left(x\right)\mathrm{d}x+D\int\left[ \tilde{q}\left(x\right)\Delta\frac{\delta\tilde{\mathcal{U}}}{\delta \tilde{q}}\left(\tilde{q},x,t\right)\right] \mathrm{d}x \\
&+\int \left[\tilde{q}\left(x\right)\mathcal{T}q_t\left(x\right)\cdot\nabla\frac{\delta\tilde{\mathcal{U}}}{\delta \tilde{q}}\left(\tilde{q},x,t\right)+q_t\left(x\right)\mathcal{T}\tilde{q}\left(x\right)\cdot\nabla\frac{\delta\tilde{\mathcal{U}}}{\delta \tilde{q}}\left(\tilde{q},x,t\right)\right]\mathrm{d}x=0,\,\,\,  t\leq T,\\
& \tilde{\mathcal{U}}\left(\tilde{q},T\right) =\int G(x,q_T) \tilde{q}\left(x\right)\mathrm{d}x,\label{eq:hje_funal_modify}
\end{aligned}
\end{equation}
Especially, with the same initial condition $\tilde{q}_t=q_t=q$, the modified MFG-1 model \eqref{eq:loss_modified} will give the same optimal solution as the MFG-2 model 
\[
\tilde{q}_s\left(x\right)=q_s\left(x\right),\; \tilde{\phi}_s\left(x\right)=\varphi_s\left(x\right),\quad \mathrm{for}\; t\leq s\leq T.
\]
\end{theorem}
Above, we suppressed the implicit dependence  $\mathcal{U}\left(\tilde{q},s\right)\coloneqq\mathcal{U}\left(\tilde{q},s;q\left(\cdot\right)\right)$ on the entire curve $q\left(\cdot\right)$ for cleaner notations. Notice that the value function $\tilde{\mathcal{U}}$ here defined in \eqref{eq:optim_loss} is different from $\mathcal{V}$ in \eqref{eq:loss_pde}, while Theorem~\ref{thm:link} tells that we can recover the same optimal solution of the MFG-2 model by solving the easier non-coupled system \eqref{eq:mfg2_decouple}.
\begin{proof}
From Proposition~\ref{prop:hje1_funcnal}, we have using the new Hamiltonian \eqref{eq:hamil_mfg2_mod}
\[
\begin{aligned}
&\partial_{t}\tilde{\mathcal{U}}\left(\tilde{q},q,t\right) =\tilde{\mathcal{H}}_{2}\left(\tilde{q},-\frac{\delta\tilde{\mathcal{U}}}{\delta\tilde{q}}\left(\tilde{q},x,t\right);q_t,\varphi_t\right)\\
=&\int \left[H_{1}\left(x,-\frac{\delta\tilde{\mathcal{U}}}{\delta\tilde{q}};q_t\right)-D\Delta\frac{\delta\tilde{\mathcal{U}}}{\delta \tilde{q}}-\tilde{F}\left(x,q_t,\varphi_t\right)\right]\tilde{q}\left(x\right)\mathrm{d}x.\\
\end{aligned}
\]
This gives the HJE \eqref{eq:hje_funal_modify} using the explicit expression for $H_1$ and $\varphi_t=-\frac{\delta\tilde{\mathcal{U}}}{\delta\tilde{q}}\left(\tilde{q}_t,x,t\right)$ from Lemma~\ref{lem:ctrl_equiv}.

Next, Since $\left(q_s,\varphi_s\right), \, s\leq t\leq T$ is given by the optimal solution of \eqref{eq:hamil2}, we have
\[
\begin{aligned}
\partial_{s}q_{s}&=-\nabla\cdot\left[\left(\mathcal{T}q_{s}+\nabla\varphi_{s}\right)q_{s}\right] +D\Delta q_{s}\\
&=-\nabla\cdot\left[\left(\mathcal{T}q_{s}+\nabla\tilde{\phi}_{s}\right)q_{s}\right] +D\Delta q_{s}.
\end{aligned}
\]
with $\tilde{\phi}_s=\varphi_s, \, s\leq t\leq T$ from Lemma~\ref{lem:ctrl_equiv}. Then the identical solutions for  $\tilde{q}_s$, $q_s$ can be found directly by comparing with the equation  \eqref{eq:fpe2_de} with the same initial value and by the uniqueness of the linear equation \eqref{eq:fpe2_de}.
\end{proof}

\section{Computational strategies for the MFG models}\label{sec:comp_strategy}
With the explicit formulations for the MFG models, we develop practical computational algorithms for solving the MFG equations to recover the optimal solution and control. Especially with the developed link between the MFG-1 and MFG-2 models, we can solve the coupled non-separable MFG-2 system based on an iterative strategy using a modified form of the decoupled MFG-1 system.

\subsection{Practical choices of the cost functions}
First, we propose the cost (activation) functional  to be optimized in the control problems in the following form
\begin{equation}
\mathcal{G}\left(q_{T}\right)+\int_{t}^{T}\mathcal{L}\left(q_{s},\alpha_{s}\right)\mathrm{d}s = 
\mathcal{G}\left(q_{T}\right) +\int_{t}^{T}\left[\int \tL\left(q_s,\alpha_s\right)\mathrm{d}x+\mathcal{F}\left(q_s\right)\right]\mathrm{d}s,\label{eq:cost}
\end{equation}
where $\mathcal{G}\left(q\right)$ quantifies the \emph{terminal error} in the final
target state and $\mathcal{L}\left(q,\alpha\right)$ is the
\emph{running loss} to characterize the cost along the control process, which is further decomposed to the cost on the control forcing $\tL$ and the cost on the running state $\mathcal{F}$. 
First, the control cost $\tL\left(q,\alpha\right)$ is set as
\begin{equation}
\tL \left(q,\alpha\right)=L(\alpha)q, \quad L(\alpha)=\frac{1}{2}\left|\alpha\right|^{2}.\label{eq:cost_l}
\end{equation}
The term regularizes the strength of the control
effect $\alpha$ according to the flow measure $q$. In particular,
in the region with a large value of $q$, a condensed particle concentration
(in term of probability measure) or a strongly turbulent fluid field
(in term of the flow vorticity) is implied. Thus the average kinetic cost for this region is the quadratic function of the control $\vert\alpha\vert^2$ weighted by the particle concentration $q$. 

Specifically, the functional $\mathcal{F}\left(q\right)$ calibrates the
energy fluctuations away from the initial and final states during the entire control process, and $\mathcal{G}\left(q\right)$ calibrates the difference at the final time step $t=T$. It is natural to require that the flow energy
cannot deviate too large away from the starting initial state $Q_i$ so that the
level of turbulence is maintained in a controlled level, while it should also approach the final target state $Q_f$ in a rapid rate. Therefore, we consider the following two common choices of the functionals according to the initial and final states $Q_i\left(x\right)$ and $Q_f\left(x\right)$:
\begin{itemize}
\item \emph{$L_{2}$-distance}: the cost functionals compute the mean square deviation from initial and final target
state using the linear combination $q^{\prime}=\gamma\left(q-Q_i\right)+\left(1-\gamma\right)\left(q-Q_f\right)$ with $0\leq\gamma\leq 1$ and $u^{\prime}=\mathcal{T}q^{\prime}$, that is,
\addtocounter{equation}{0}\begin{subequations}\label{eq:cost1} 
\begin{align}
\mathcal{F}\left(q\right) =\frac{1}{2}\int\left|u^{\prime}\right|^{2}\mathrm{d}x =\frac{1}{2}\int\left|\mathcal{T}q^{\prime}\right|^{2}\mathrm{d}x,\quad&\mathrm{with\;}F\left(q,x\right)\coloneqq\frac{\delta\mathcal{F}}{\delta q}=\mathcal{T}^{*}\cdot\mathcal{T}q^{\prime},\label{eq:cost_f1}\\
\mathcal{G}\left(q_{T}\right) =\frac{1}{2}\int\left|q_{T}-Q_{f}\right|^{2}\mathrm{d}x,\quad&\mathrm{with\;}G\left(q,x\right)\coloneqq\frac{\delta\mathcal{G}}{\delta q}=q_{T}-Q_{f}.\label{eq:cost_g1}
\end{align}
\end{subequations}
\item \emph{KL-divergence}: since $q$ can be viewed as the probability density of the Lagrangian tracer field, the cost functionals can compare the KL-divergence (relative entropy) with the distributions from the linear combination of the initial and final target states $\bar{Q}=\gamma Q_i+\left(1-\gamma\right)Q_f$ with $0\leq\gamma\leq 1$, that is
\addtocounter{equation}{0}\begin{subequations}\label{eq:cost2} 
\begin{align}
\mathcal{F}\left(q\right) =D_{\mathrm{KL}}\left(q,\bar{Q}\right) =\int q \log\frac{q}{\bar{Q}}\mathrm{d}x,\;\;&\mathrm{with\;}F\left(q,x\right)\coloneqq\frac{\delta\mathcal{F}}{\delta q}=1+\log\frac{q}{\bar{Q}},\label{eq:cost_f2}\\
\mathcal{G}\left(q_{T}\right) =D_{\mathrm{KL}}\left(q_T,Q_f\right)=\int q_T \log\frac{q_T}{Q_f}\mathrm{d}x,\;\;&\mathrm{with\;}G\left(q,x\right)\coloneqq\frac{\delta\mathcal{G}}{\delta q}=1+\log\frac{q_T}{Q_f}.\label{eq:cost_g2}
\end{align}
\end{subequations}
\end{itemize}

\subsection{Algorithms to solve the decoupled MFG-1 model}
Here, we propose computational strategies to solve the MFG models \eqref{eq:loss_tracer1} and \eqref{eq:loss_pde}. Equivalently, we can develop ensemble-based control strategies based on the stochastic control models \eqref{eq:loss_sde_tracer1} and \eqref{eq:loss_pde1}.
First, for the MFG-1 model \eqref{eq:loss_tracer1}, the optimal solution can be found by separately solving the decoupled forward and backward MFG equations \eqref{eq:mfg1}.
\begin{algorithm}[H]
\begin{algorithmic}[1] 
\Ensure{Given the flow vorticity field $q_{s}\left(x\right), s\in \left[t, T\right]$ and the initial   tracer density $\rho\left(x\right)$}
\State{solve the HJE \eqref{eq:hje1}  backwardly in time to get the function $\phi_s\left(x\right)$ with $\phi_T\left(x\right)=-G\left(x,q_T\right)$ at $s=T$.}
\State{recover the control for the entire time window $\alpha_s\left(x\right)=\nabla\phi_s\left(x\right),t\leq s\leq T$.}
\State{solve the continuity equation \eqref{eq:fpe1} forward in time using the optimal control $\alpha_s\left(x\right)$ to get the optimal tracer density starting from the initial configuration $\rho_t\left(x\right)=\rho\left(x\right)$.}
\end{algorithmic}
\caption{MFG-1 model for optimal control of tracer field using PDE model \label{alg:mfg1}}
\end{algorithm}
Notice that above in solving the forward equation \eqref{eq:fpe1}, an alternative approach is to adopt the SDE formulation \eqref{eq:loss_sde_tracer1} thus to get the approximated optimal density solution $\rho_s$ through an Monte-Carlo ensemble method. This approach could be advantageous in practical problems to efficiently recover the density distribution.
Therefore, we can propose the ensemble-based algorithm corresponding to Algorithm~\ref{alg:mfg1}.
\begin{algorithm}[H]
\begin{algorithmic}[1] 
\Ensure{Given the flow vorticity field $q_{s}\left(x\right), s\in \left[t, T\right]$ and the initial   tracer density $\rho\left(x\right)$}
\Require{generate random samples $X_{t}^{i}\sim\rho\left(x\right),i=1,\cdots,N$ according to the tracer density.}
\State{solve the HJE \eqref{eq:hje1}  backwardly in time to get the function $\phi_s\left(x\right)$ with $\phi_T\left(x\right)=-G\left(x,q_T\right)$ at $s=T$.}
\State{recover the control $\alpha_s\left(x\right)=\nabla\phi_s\left(x\right),t\leq s\leq T$ for the entire time window.}
\State{solve the SDE \eqref{eq:loss_sde_tracer1} forward in time independently for each sample trajectory $X_s^{i},t\leq s\leq T$.}
\State{recover the optimal tracer density using the empirical distribution $\rho^{N}_{s}\left(x\right)=\frac{1}{N}\sum_{i}\delta_{X_s^{i}}\left(x\right)$.}

\end{algorithmic}
\caption{MFG-1$'$ model for controlling tracer field using finite ensemble model\label{alg:mfg1_sde}}
\end{algorithm}
At the limit $N\rightarrow\infty$, Algorithm~\ref{alg:mfg1_sde} for finite ensemble approximation converges to the original Algorithm~\ref{alg:mfg1} at the continuum limit. This ensemble approach provides an effective alternative strategy to recover the target optimal tracer density without running the usually more expensive Fokker-Planck equation especially in the higher-dimensional cases. The ensemble approach will become more useful next for controlling the transition in flow states.

\subsection{Iterative algorithm to solve the coupled MFG-2 model}
Recall that the optimal solution for MFG-2 model solves the coupled MFG-2 system \ref{eq:mfg2}.
In solving the MFG-2 model \eqref{eq:loss_pde}, we need to deal with the coupled forward and backward equations \eqref{eq:mfg2}   together. To address this difficulty, we first solve a separable MFG according to Algorithm~\ref{alg:mfg1} or \ref{alg:mfg1_sde}. Then the optimal solution is achieved by an iterating approach as finding a fixed point.

We introduce  a `push forward' map, $\left(\tilde{q}_s,\tilde{\alpha}_s\right)= \mathcal{P}\left(q_s,\alpha_s\right)$, based on the solution of the modified MFG-1 model \eqref{eq:mfg2_decouple} with the given solution of $\left\{q_s,\alpha_s\right\},t\leq s\leq T$
\addtocounter{equation}{0}\begin{subequations}\label{eq:mfg2_modified} 
\begin{align}
\partial_{s}\tilde{q}_{s} &+\nabla\cdot\left[\left(\mathcal{T}q_{s}+\nabla\tilde{\phi}_{s}\right)\tilde{q}_{s}\right]  =D\Delta \tilde{q}_{s},\label{eq:fpe2_m}\\
\partial_{s}\tilde{\phi}_{s} &+\frac{1}{2}\left|\nabla\tilde{\phi}_{s}\right|^{2}+\mathcal{T}q_{s}\cdot\nabla\tilde{\phi}_{s} + D\Delta\tilde{\phi}_{s}  =F(x,q_s)-\mathcal{T}^{*}\cdot\left(q_{s}\alpha_{s}\right), \label{eq:hje2_m}
\end{align}
\end{subequations}
with the initial and final conditions $\tilde{q}_{t}\left(x\right)=q\left(x\right)$ and $\tilde{\phi}_{T}\left(x\right)=-G(x,q_T)$, and the new control $\tilde{\alpha}_s=\nabla \tilde{\phi}_s$. 
The above equations \eqref{eq:mfg2_modified} define a separable problem that is easy to solve individually. 
From Theorem~\ref{thm:link}, if we can find a fixed point of the map, $\left(q_s,\alpha_s\right)=\mathcal{P}\left(q_s,\alpha_s \right)$, based on the modified MFG-1 model \eqref{eq:mfg2_modified}, the fixed point solution provides the solution for the MFG-2  system \eqref{eq:loss_pde}. Therefore, we expect solution of \eqref{eq:mfg2_modified} gives good approximation to the optimal solution of the MFG-2 model given a close estimate of the input states $\left\{q_s,\alpha_s\right\}$ (see Proposition~\ref{prop:stability}). If we further assume the solution of \eqref{eq:mfg2} is unique, we can solve the MFG-2 model through an iterative algorithm.

Based on the above consideration, we propose the following iterative strategy aiming to minimize the value function $\mathcal{I}\left(q_s,\alpha_s\right)$ of MFG-2 model.
Let $\{\tilde{q}^{(n+1)}_s,\tilde{\alpha}^{(n+1)}_s\}$ be the solution of \eqref{eq:mfg2_modified} using the input functions $\{q^{(n)}_s,\alpha^{(n)}_s\}$ (from the previous iteration step). First, we construct a new state $q^{\mu}_{s}$ as the linear interpolation of the two functions
\begin{equation}\label{eq:update_q}
q^{\mu}_s=\mu q^{(n)}_s + \left(1-\mu\right)\tilde{q}^{(n+1)}_s.
\end{equation}
with $0\leq \mu\leq1$. Then, the corresponding new interpolation state of $\alpha^{\mu}_s$ is constructed based on the consistency with the continuity equation such that
\[
\partial_{s}q^{\mu}_{s}  +\nabla\cdot\left[\left(\mathcal{T}q^{\mu}_{s}+\alpha^{\mu}_{s}\right)q^{\mu}_{s}\right]=D\Delta q^{\mu}_{s}.
\]
This can be achieved by comparing the equations for $\{ \tilde{q}_s^{(n+1)},\tilde{\alpha}_s^{(n+1)}\}$ and $\{ q_s^{(n)},\alpha_s^{(n)}\}$. The solution can be found to follow the relation
\begin{equation}
\alpha_{s}^{\mu} = \frac{\mu\alpha_s^{(n)}q_s^{(n)} + \left(1-\mu\right)\tilde{\alpha}_s^{(n+1)}\tilde{q}_s^{(n+1)}}{\mu q_s^{(n)} + \left(1-\mu\right)\tilde{q}_s^{(n+1)}} +\left(1-\mu\right)\left(\mathcal{T}q_s^{(n)}-\mathcal{T}\tilde{q}_s^{(n+1)}\right).\label{eq:op_comb}
\end{equation}
Finally,  the updated state is defined by the optimal $\mu$  at the point that minimizes the value of the cost function \eqref{eq:loss_pde} of the MFG-2 model, that is,
\begin{equation}
q^{(n+1)}_s \coloneqq q^{\mu^{*}}_s,\; \alpha^{(n+1)}_s \coloneqq \alpha^{\mu^{*}}_s, \quad \mathrm{with}\:\mu^* = \underset{0\leq\mu\leq 1}{\mathrm{argmin}}\: \mathcal{I}\left(q_{s}^{\mu},\alpha_{s}^{\mu}\right).
\end{equation}
This finishes the one-step updating for the fixed point iteration. We describe the algorithm using the iterative method to find the optimal control solution as follows.

\begin{algorithm}[H]
\begin{algorithmic}[1] 
\Require{set up the initial flow vorticity $q_s^{(0)}$ and control functions $\alpha_s^{(0)}$.}

\For{$n\leq N_{\max}$ while $d\left(q_s^{(n)},q_s^{(n-1)}\right)>\epsilon$ or  $d\left(\alpha_s^{(n)},\alpha_s^{(n-1)}\right)>\epsilon$}
\State{solve the separable equations \eqref{eq:mfg2_modified} to get the new states $\{ \tilde{q}_s^{(n+1)},\tilde{\alpha}_s^{(n+1)}\} $ with input $\{ q_s^{(n)},\alpha_s^{(n)}\}$.}
\State{find the optimal $\mu^*$ that minimizes the cost function $\mathcal{I}\left(q_{s}^{\mu},\alpha_{s}^{\mu}\right)$ by line searching with a sequence of $\mu_{i}=\frac{i}{L},i=1,\cdots,L-1$ using the combination of the two states \eqref{eq:update_q} and \eqref{eq:op_comb}.}
\State{update the next states $q_s^{(n+1)}=q_s^{\mu^*}$ and $\alpha_s^{(n+1)}=\alpha_s^{\mu^*}$.}

\EndFor
\end{algorithmic}
\caption{MFG-2 model for optimal control of the flow field\label{alg:mfg2}}
\end{algorithm}

\subsubsection*{A sufficient condition for the stability of the iterative scheme}
In Algorithm~\ref{alg:mfg2}, we need to take the critical step to update the next stage solution using the optimal linear combination \eqref{eq:update_q} instead of directly taking the solution from the equations \eqref{eq:mfg2_modified}. This is confirmed by the following necessary condition describing the stability of the iteration scheme.
\begin{claim}\label{prop:stability}
With the cost functions defined in \eqref{eq:cost1} or \eqref{eq:cost2}, the value of the target cost function in \eqref{eq:loss_pde} is not guaranteed to decrease  if the solution of the equations \eqref{eq:mfg2_modified} is directly applied to update the next stage. That is, we may have 
\[
\mathcal{I}\left(\tilde{q}_s,\tilde{\alpha}_s\right) > \mathcal{I}\left(q_s,\alpha_s\right),
\]
with $\mu=0$ during some iteration steps.
On the other hand, one can expect to reduce the value of the cost function from the input $\left(q_s,\alpha_s\right)$ in every iteration,
\begin{equation}\label{eq:decay_loss}
\mathcal{I}\left(q^{\mu}_s,\alpha^{\mu}_s\right) \leq \mathcal{I}\left(q_s,\alpha_s\right),
\end{equation}
by taking some $\mu> 0$ in updating the solution in \eqref{eq:update_q}.

\end{claim}
\begin{proof}
Denote the value functional to be optimized above as 
\[
\mathcal{J}\left(\tilde{q}_s,\tilde{\alpha}_s;q_s,\alpha_s\right) = \int G(x,q_T) \tilde{q}_{T} \left(x\right)\mathrm{d}x +\int_{t}^{T}\left\{ \int \left[L\left(\tilde{\alpha}_{s}\right) +\tilde{F}\left(x,q_{s},\alpha_s \right)\right]\tilde{q}_{s}\left(x\right)\mathrm{d}x\right\} \mathrm{d}s.
\]
First, since $\left(\tilde{q}_s,\tilde{\alpha}_s\right)$ is the optimal solution of the problem \eqref{eq:loss_modified}, they minimize the corresponding value function
\begin{equation}\label{eq:iter1}
\mathcal{J}\left(\tilde{q}_s,\tilde{\alpha}_s;q_s,\alpha_s\right)\leq \mathcal{J}\left(q_s,\alpha_s;q_s,\alpha_s\right),
\end{equation}
as long as we choose $\left(q_s,\alpha_s\right)$ also satisfying the continuity equation in \eqref{eq:fpe2}. 

On the other hand, the target loss function to minimize has the form
\[
\mathcal{I}\left(q_s,\alpha_s\right) =\mathcal{G}\left(q_{T}\right)+\int_{t}^{T}\left[ \int L\left(\alpha_{s}\right)q_{s}\left(x\right)\mathrm{d}x +\mathcal{F}\left(q_{s}\right)\right]\mathrm{d}s.
\]
To evaluate the effect from one-step update, we define the improvement function
\begin{equation}\label{eq:incre}
\begin{aligned}
g\left(\mu\right) = &\; \mathcal{I}\left(q^{\mu}_s,\alpha^{\mu}_s\right) - \mathcal{I}\left(q_s,\alpha_s\right)\\
\leq & \left[\mathcal{G}\left(q^{\mu}_{T}\right)-\int G(x,q_T) \tilde{q}_{T} \right]+\int_{t}^{T}\left[ \mathcal{F}\left(q^{\mu}_{s}\right)-\int F\left(x,q_{s} \right)\tilde{q}_{s}\right]\mathrm{d}s\\
+&\left[\int G(x,q_T) q_{T} -\mathcal{G}\left(q_{T}\right)\right]+\int_{t}^{T}\left[ \int F\left(x,q_{s} \right)q_{s}-\mathcal{F}\left(q_{s}\right)\right]\mathrm{d}s\\
+&\int \left[L\left(\alpha^{\mu}_{s}\right)q^{\mu}_{s}-L\left(\tilde{\alpha}_{s}\right)\tilde{q}_{s}\right] + \int_{t}^{T}\left[ \int \left(q_s\alpha_s \right)\cdot \mathcal{T}\left(\tilde{q}_s-q_{s}\right)\right]\mathrm{d}s.
\end{aligned}
\end{equation}
Above in the inequality, \eqref{eq:iter1} is used to link the solutions from the modified MFG-1 model \eqref{eq:mfg2_modified} with the updated new state $q^{\mu}_s=\mu q_s + \left(1-\mu\right)\tilde{q}_s$. 

First from \eqref{eq:update_q} and \eqref{eq:op_comb} with $\mu=1$, we have $q^{\mu}_s=q_s,\alpha^{\mu}_s=\alpha_s$, thus
\[g\left(1\right)=\mathcal{I}\left(q_s,\alpha_s\right) - \mathcal{I}\left(q_s,\alpha_s\right)=0.\]
Then at the other end point with $\mu=0$, that is, $q^{\mu}_s=\tilde{q}_s,\alpha^{\mu}_s=\tilde{\alpha}_s$, we can compute the explicit expressions for the right hand side of \eqref{eq:incre} based on the explicit loss functions. Also since the terminal and running costs $\mathcal{G}$ and $\mathcal{F}$ share similar forms in both $L_2$ and KL-divergence cases, we compute the terminal cost below and the running cost will follow with very similar expressions. For the $L_2$ loss \eqref{eq:cost1}, we find
\[
\left[\mathcal{G}_{L_2}\left(\tilde{q}_{T}\right)-\int G_{L_2}(x,q_T) \tilde{q}_{T} \right]+\left[\int G_{L_2}(x,q_T) q_{T} -\mathcal{G}_{L_2}\left(q_{T}\right)\right]
=\frac{1}{2}\int \left(\tilde{q}_T-q_T\right)^2\geq 0,
\]
and for the KL-divergence loss \eqref{eq:cost2}
\[
\left[\mathcal{G}_{KL}\left(\tilde{q}_{T}\right)-\int G_{KL}(x,q_T) \tilde{q}_{T} \right]+\left[\int G_{KL}(x,q_T) q_{T} -\mathcal{G}_{KL}\left(q_{T}\right)\right]
=\int \tilde{q}_T\log\frac{\tilde{q}_T}{q_T}\geq 0.
\]
In the last line of \eqref{eq:incre}, the first term will vanish at $\mu=0$ since $q_s^0=\tilde{q}_s, \alpha_s^0=\tilde{\alpha}_s$. However, the sign of the second term with the integrant, $\left(q_s\alpha_s \right)\cdot \mathcal{T}\left(\tilde{q}_s-q_{s}\right)$, becomes indefinite and could frequently reach positive values during the iterations.
Therefore,  a positive value could be reached on the right hand side of \eqref{eq:incre} using both loss functions. Thus, the total cost $\mathcal{I}$ is not guaranteed to decrease with direct update using a constant $\mu=0$ (see Fig.~\ref{fig:cost} as a confirmation from direct numerical tests).

Finally, notice that using the specific loss functions \eqref{eq:cost1} or \eqref{eq:cost2}, the the function, $g\left(\mu\right) = A\mu^2+B\mu+C+O\left(\Vert \tilde{q}_s-q_{s}\Vert^2\right)$, can be expressed as a quadratic function about $\mu$ in the leading order expansion.  Further, it can be checked the coefficient before the $\mu^2$ term is positive. Thus to ensure there exists some  $\mu^*$ such that $g\left(\mu^{*}\right)<0$, we only need to consider the case with $g\left(0\right)>0$. In addition, taking into account $g\left(1\right)=0$, the property of the quadratic function suggests that  negative values of $g\left(\mu\right)$ will be reached with some $\mu>0$ unless the critical case with minimun reached at $\mu=1$.
\end{proof}

\section{Model performance using prototype test models}\label{sec:numerics}
In this section, we demonstrate the performance of the proposed MFG models through detailed numerical experiments. A prototype nonlinear advection-diffusion model modified from the viscous Burger's equation is used as a test example simulating multiscale vortical flows. This simple model preserves many key properties of the more general turbulent models appearing in many fields, while provides a clean tractable setup for confirming the important basic properties discussed in the previous sections.

\subsection{The modified Burger's equation as a prototype test model}\label{subsec:mvb}

In the numerical tests, we introduce the control problem for the 
modified viscous Burger's (MVB) equation 
\begin{equation}
\partial_{s}q_{s}\left(x\right)+\nabla\cdot\left[u_{s}\left(x\right)q_{s}\left(x\right)\right]=\nu\Delta q_{s}\left(x\right) +f_s\left(x\right).\label{eq:burgers}
\end{equation}
with the control forcing in the form $f_s = -\nabla\cdot\left(\alpha_{s}q_{s}\right)$. The solution is defined on either the one-dimensional periodic domain, $x\in\left[-L,L\right]$, or the two-dimensional doubly periodic domain, $(x,y)\in\left[-L,L\right]\times\left[-L,L\right]$. 
The vorticity field $q_s$ can be projected on the Fourier modes then the advection velocity field $u_s$ can be expressed explicitly under the spectral representation based on each wavenumber $k=\left(k_{x},k_{y}\right)^{\intercal}$ as
\[
u_{s}\left(x\right)=\mathcal{T}q_{s}\left(x\right)=\sum_{\vert k\vert\neq0}i\vert k\vert^{-2}\left(k_{y},-k_{x}\right)^{\intercal}\hat{q}_{k}\left(s\right)e^{ikx},\quad\mathrm{and}\; q_{s}\left(x\right)=\sum_{k}\hat{q}_{k}\left(s\right)e^{ikx}.
\]
In the one-dimensional case, the velocity field reduces to the simpler form $u_{s}\left(x\right)=\sum_{k\neq0}ik^{-1}\hat{q}_{k}\left(s\right)e^{ikx}$.
In fact, it is direct to check that if $u_s\left(x\right)$ satisfies the solution of the viscous Burger's equation, $q_s=-\partial_xu_s$ gives the solution of \eqref{eq:burgers}. Therefore,  a sequence of exact steady state solutions can be constructed for the MVB equation \eqref{eq:burgers} based on the explicit analytic solutions of $u_s$ \cite{boritchev2014decaying} 
\begin{equation}
Q_{\sigma,a}\left(x\right)=2\nu\sigma^{2}\mathrm{sech}^2\left(\sigma\vert x-a\vert)\right),\label{eq:steady_mvb}
\end{equation}
with two parameters $\sigma,a$. The steady solutions indicate the persistent coherent structures in general turbulent flow fields.
Therefore, we design the control problem for recovering the optimal control forcing $f_s$ during $s\in\left[0,T\right]$ driving between two steady solutions
$Q_{i}\left(x\right)=Q_{\sigma_{,}-L/2}\left(x\right)$ and $Q_{f}\left(x\right)=Q_{\sigma,L/2}\left(x\right)$. The cost
functionals in the optimal control are following the two typical examples in \eqref{eq:cost1} and \eqref{eq:cost2}.
It shows that many representative features of general interests including multiscale turbulent behavior and extreme events \cite{abramov2003hamiltonian,majda2000remarkable} can be generated in the simplified MVB model \eqref{eq:burgers}.

In the numerical experiments, the MVB equation is solved by a pseudo-spectral
method with a Galerkin truncated spectral representation of $J=256$ modes in both $x$ and $y$ directions. The finite truncation model is suitable for more general applications with explicit multiscale turbulent structures. The equation is integrated in time
by an explicit-implicit Runge-Kutta method with the implicit part
only for the dissipation term. Model parameters in the numerical tests are listed
below in Table \ref{tab:Parameters}. A typical steady state solution of the MVB equation is plotted in Fig.~\ref{fig:steady} showing a coherent flow structure.
In the following, we first consider the simpler one-dimensional case giving a detailed discussion on both the MFG-1 model \eqref{eq:loss_tracer1}
for the control of the transport of passive tracers, and the performance of the iterative strategy solving the MFG-2
model \eqref{eq:loss_pde} for the control of the flow vorticity equation. Then, the control performance of the control model on the more complex two-dimensional flow equations is tested.

\begin{table}
\centering
\begin{tabular}{cccccccccc}
\toprule 
$\nu$ & $L$ & $J$ & $\Delta t$ & $T$ & & $\gamma$ & $\sigma$ & $a$ & \tabularnewline
\midrule
\midrule 
0.5 & 10 & 256 & $1\times10^{-3}$ & 10 & & 0.2 & 1 & $\pm L/2$& \tabularnewline
\bottomrule
\end{tabular}
\caption{Parameters used for the MVB test model. The last there columns show control parameters for the initial and final target states \eqref{eq:steady_mvb}.\label{tab:Parameters}}
\end{table}
\begin{figure}
\includegraphics[scale=0.35]{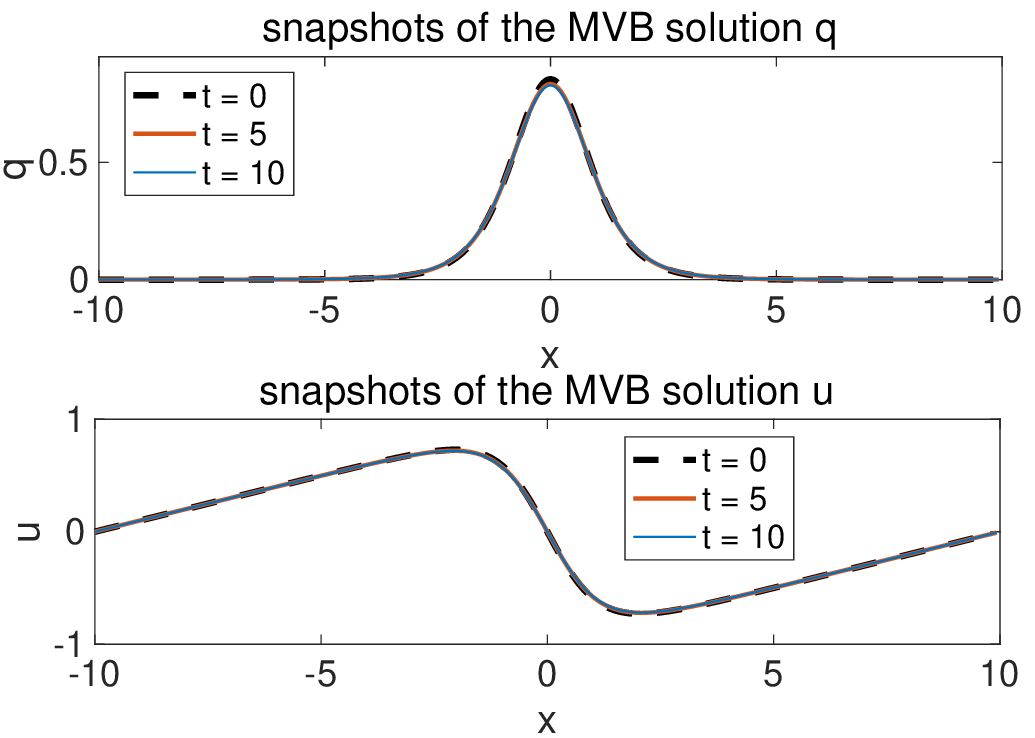}
\includegraphics[scale=0.35]{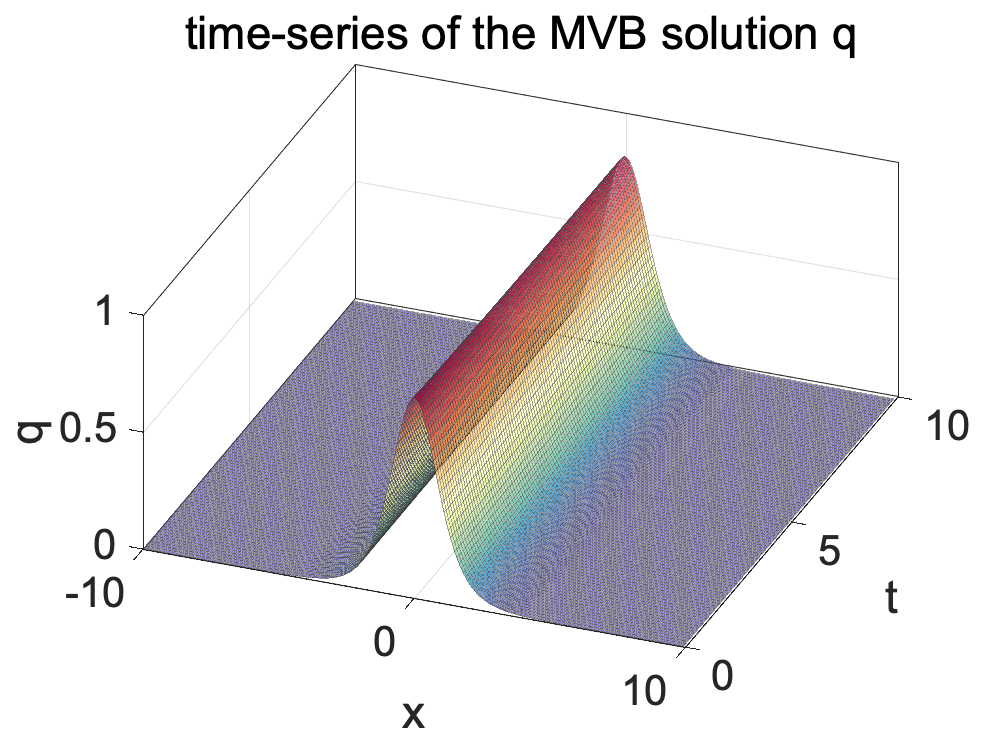}
\includegraphics[scale=0.35]{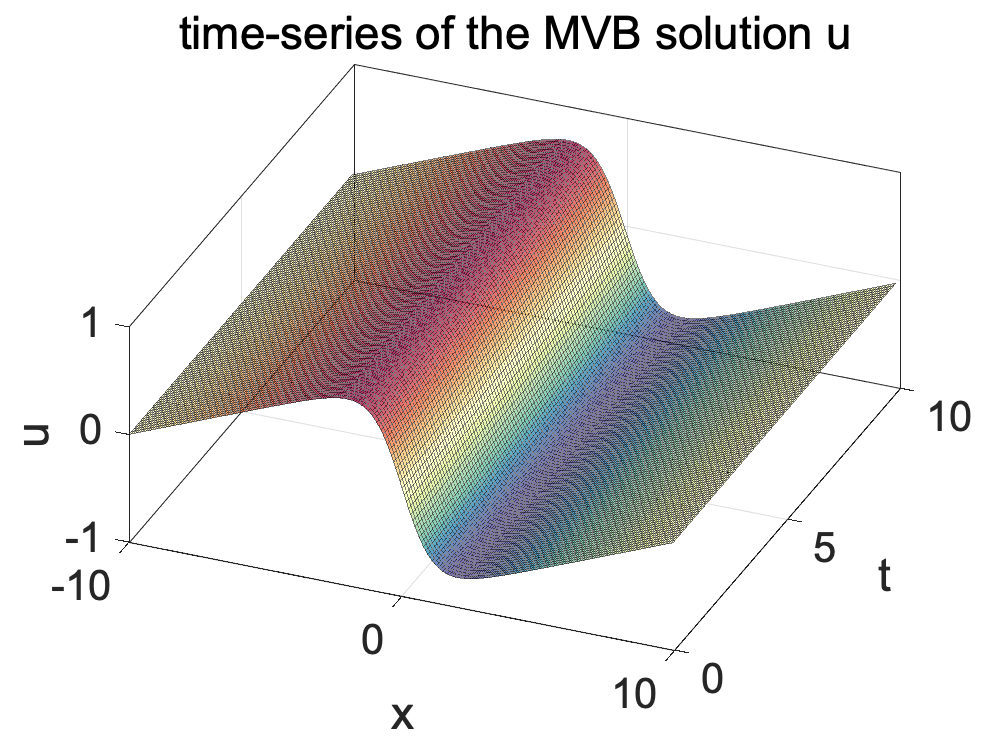}

\caption{An illustration of the steady solution of the MVB equation with $q_t=-\partial_x u_t$.}\label{fig:steady}

\end{figure}

\subsection{MFG-1 model for tracer field control}

First, we show the performance of the tracer control model in the one-dimensional case according to Algorithm~\ref{alg:mfg1} and \ref{alg:mfg1_sde}. 
The advection flow solution $u_t$ is generated by the steady state solution of \eqref{eq:burgers} with $f\equiv 0$ as shown in Fig.~\ref{fig:steady}. We set the tracer initial density as $Q_{i}\left(x\right)=Q_{1,-L/2}\left(x\right)$ and the final target density as $Q_{f}\left(x\right)=Q_{1,L/2}\left(x\right)$.
Therefore, the tracer control problem asks the optimal control action that drives the tracers across the `barrier' set by the advection flow in the center of the domain. Both the $L_2$ cost \eqref{eq:cost1} and the KL-divergence cost \eqref{eq:cost2} are applied in optimizing the cost function \eqref{eq:loss_tracer1}.

\subsubsection{Control of tracer density function with the PDE model}
The forward continuity equation and backward HJE are decoupled in the MFG-1 model case. Following Algorithm \ref{alg:mfg1}, we first solve the control action 
$\alpha_{s}=\partial_{x}\phi_{s}$ by solving the backward equation \eqref{eq:hje1}
using the prescribed solution of $q_{s}$, 
with $F$ and $G$ from the corresponding costs \eqref{eq:cost1} or \eqref{eq:cost2}. Next, we solve the forward equation \eqref{eq:fpe1} to get the controlled tracer density
field using the achieved optimal control solution $\phi_s$.
The solution gives the optimal controlled tracer density field $\rho_{s}$
together with the optimal control forcing $\alpha_{s}$ exerted on top of
the transporting flow velocity $u_{s}=\mathcal{T}q_s$.

The optimal solutions for the MFG-1 model achieved with both $L_2$ and KL-divergence cost functions are plotted in Fig.~\ref{fig:ctrl_tracer1}. Comparable control solutions are found under the different forms of cost functions. In this control problem for tracer density,  the initial tracer field concentrates on the left side of the domain at $-L/2$ while the target field lies on the right at $L/2$. With the steady advection flow field $u_s$ as shown in Fig.~\ref{fig:steady}, the passive tracers are driven toward the center of the domain at $x=0$. The control forcing is required to guide the tracers going against the tendency from the advection flow velocity. As a result, the controlled tracer field $\rho_s$ diverged into two routes in opposite directions, one travels across the center region of the domain and the other goes across the boundary using the periodic boundary condition, converged at the final target location.  A strong control forcing $\alpha_s$ is exerted at the starting time to drive the tracer density quickly toward the target, then is reduced to smaller values to balance the running cost part of the control. The tracer density fields finally reach the target state $\rho_T$ with good agreements, indicating successful control performance in both cases under different losses.

\begin{figure}
\subfloat[$L_2$ loss]{
\includegraphics[scale=0.35]{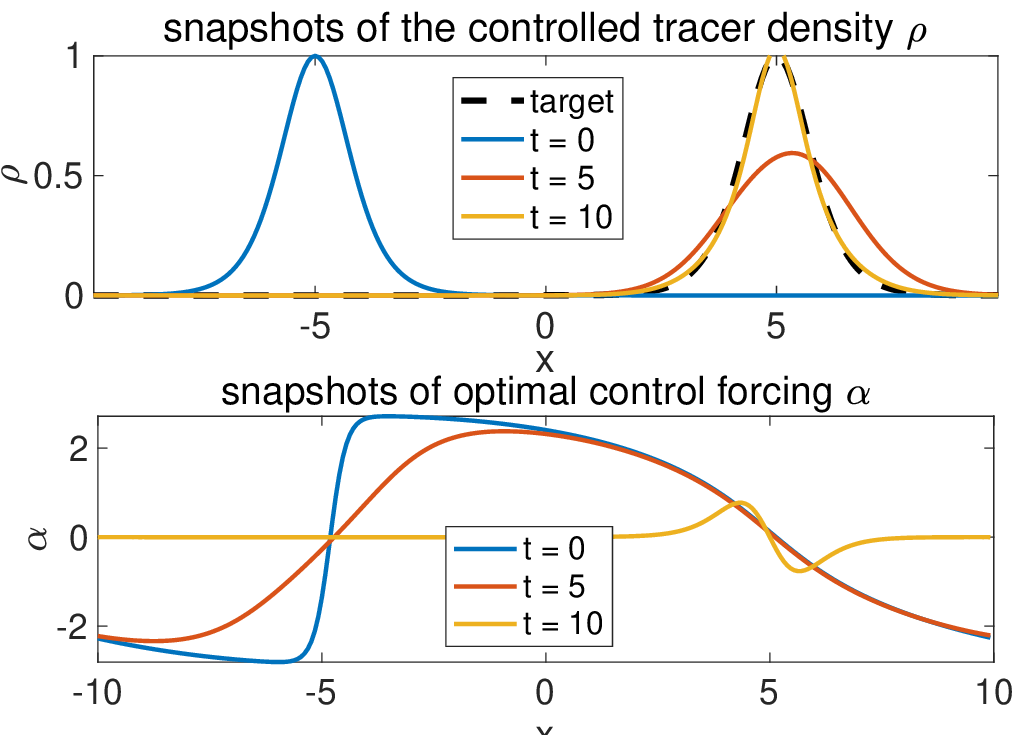}\includegraphics[scale=0.35]{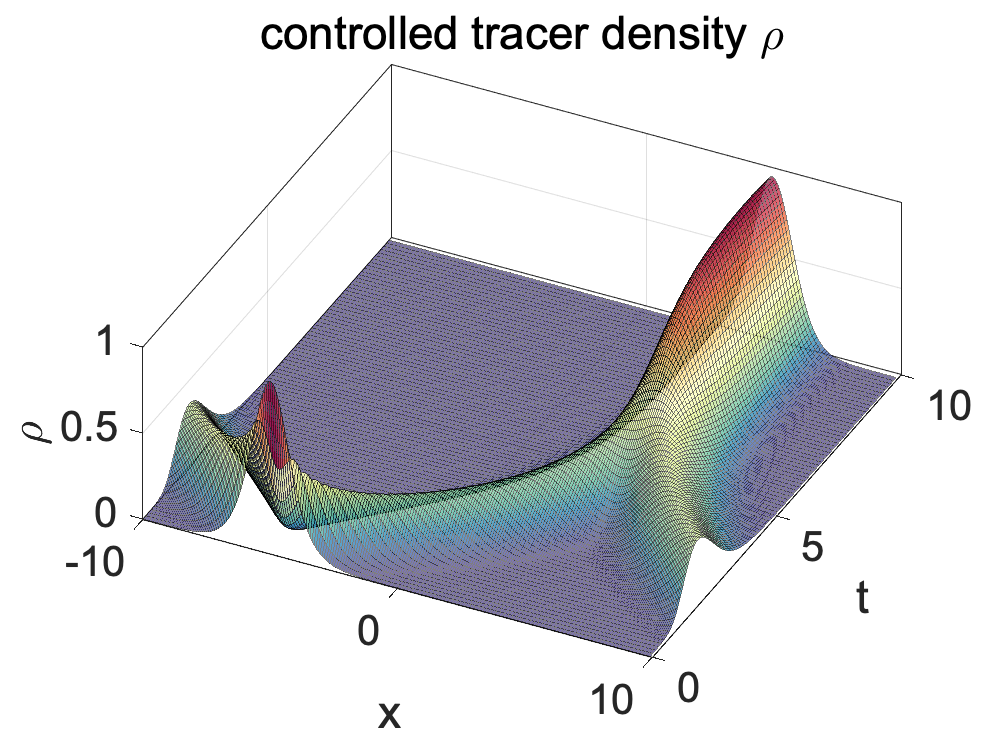}
\includegraphics[scale=0.35]{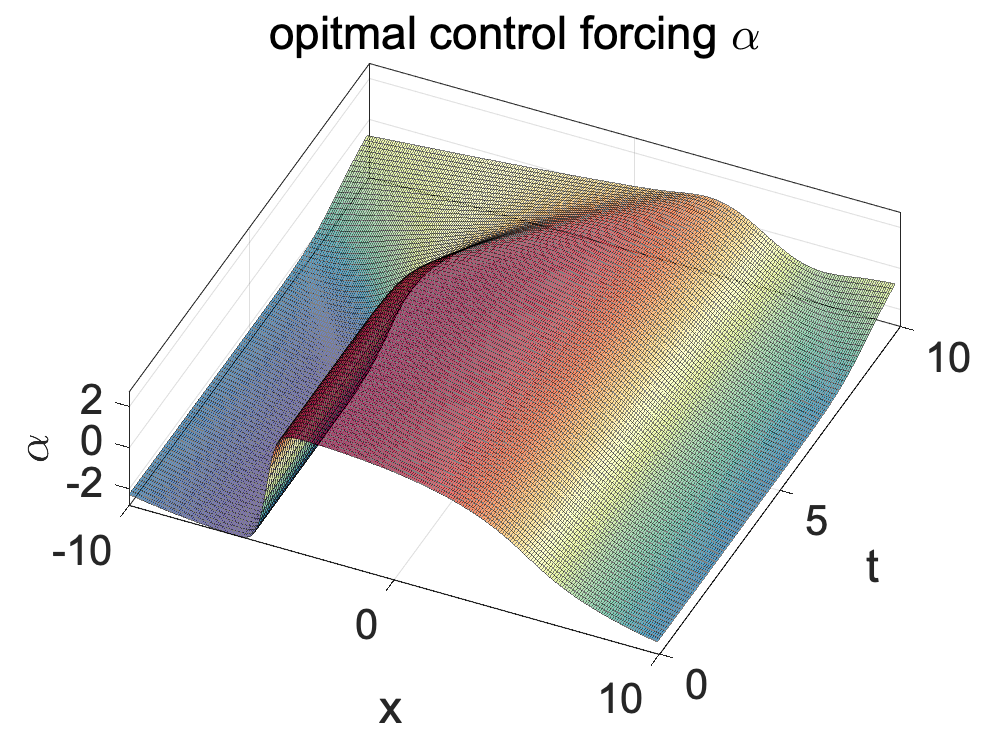}
}

\subfloat[KL-divergence loss]{
\includegraphics[scale=0.35]{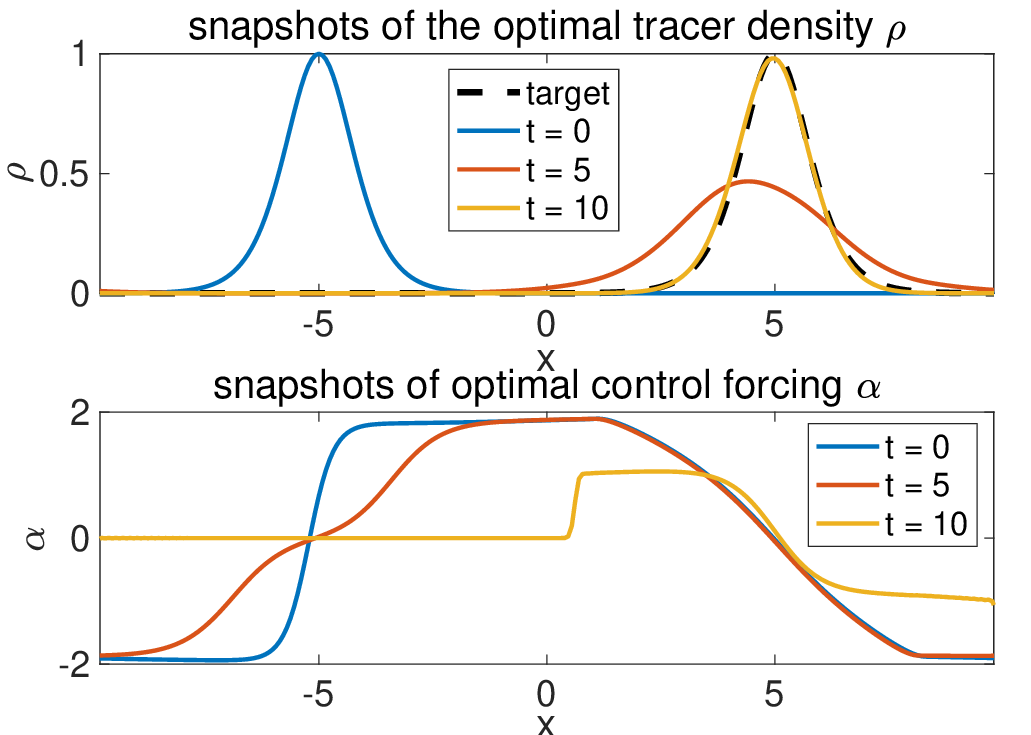}\includegraphics[scale=0.35]{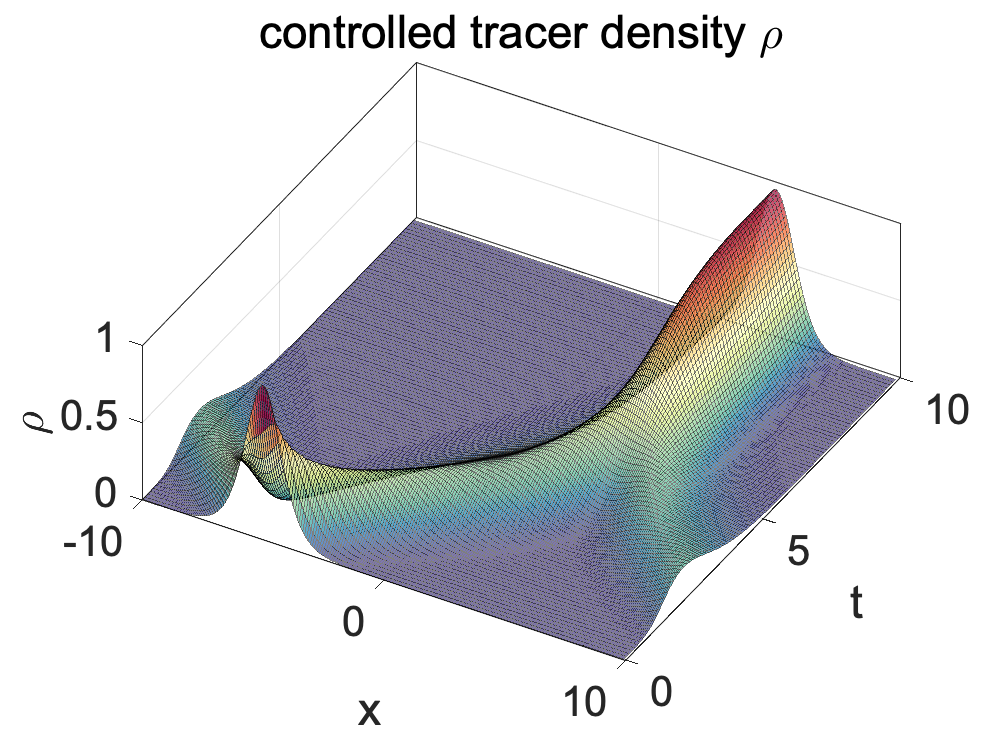}\includegraphics[scale=0.35]{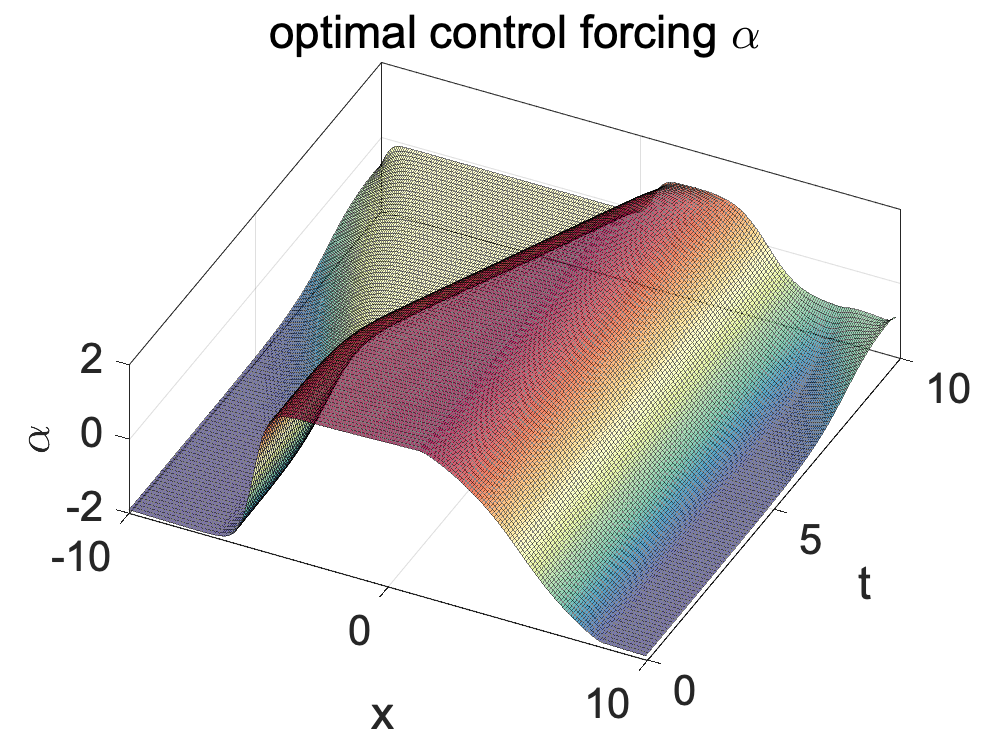}
}

\caption{Optimal controlled solutions of MFG-1 model for tracer transport with different loss functions.}\label{fig:ctrl_tracer1}
\end{figure}

\subsubsection{Control of ensemble empirical distributions with the SDE model}\label{subsec:sampling}
Corresponding, we can use the ensemble SDE approach for solving the controlled optimal tracer density through estimating the empirical measure from the samples. Especially, Algorithm~\ref{alg:mfg1_sde} for the tracer control model provides an effective way to produce samples agreeing any non-Gaussian PDFs by setting it as the targeting state $X^i_T\sim Q_f$ for the terminal tracer density. Here, the initial MC samples of the tracers  $X_0\sim \rho_0=\mathcal{N}\left(\mu_0,\sigma_0^2\right)$ can be easily sampled from a normal distribution. The optimal control forcing $\alpha_s$ is still solved from the backward HJE \eqref{eq:hje1} with the terminal condition $G$ defined by the target state $Q_f$. But instead, the forward continuity equation is solved by the MC approximation of the SDE \eqref{eq:loss_sde_tracer1}.
Then, the resulting ensemble members through the controlled SDE model sample the target PDF, $\frac{1}{N} \sum_{i}\delta_{X^{i}_T}\simeq Q_f$. This method will become more useful for sampling high dimensional PDFs including highly non-Gaussian structures.

In Fig.~\ref{fig:tracer_sampling}, we show one simple test to sample non-Gaussian PDFs in the shape of the functions \eqref{eq:steady_mvb}. Using the SDE model with $N=10^4$ samples, the initial samples are drawn from a standard normal distribution. Two target PDFs with linear tails of different extent $\sigma=0.5$, $\sigma=1$ are used. It shows that the final empirical sample distributions of the tracer particles accurately capture the non-Gaussian shapes in the PDFs, and are also agree with the PDE control model results. Deviation only appears in the long tail region due to insufficient representation of the extreme events. The ensemble representations will also applied next using the SDE model for controlling flow vorticity states. 
\begin{figure}

\includegraphics[scale=0.32]{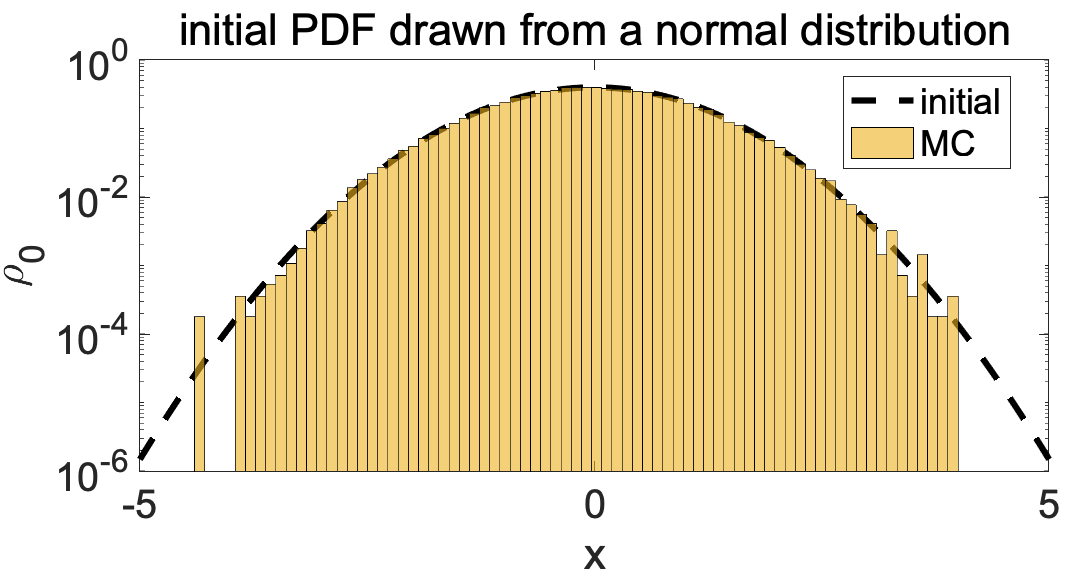}\includegraphics[scale=0.32]{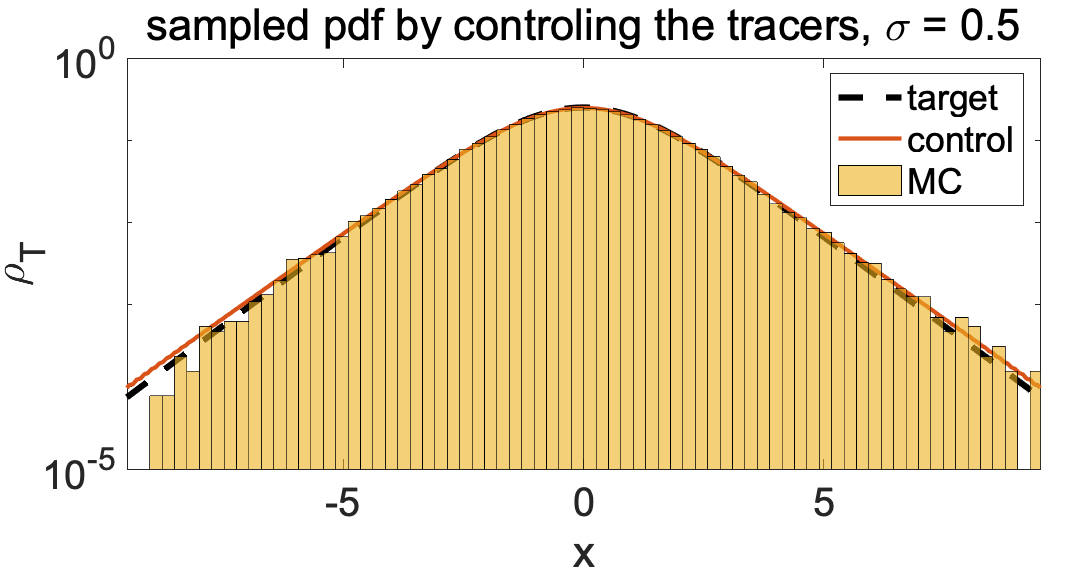}\includegraphics[scale=0.32]{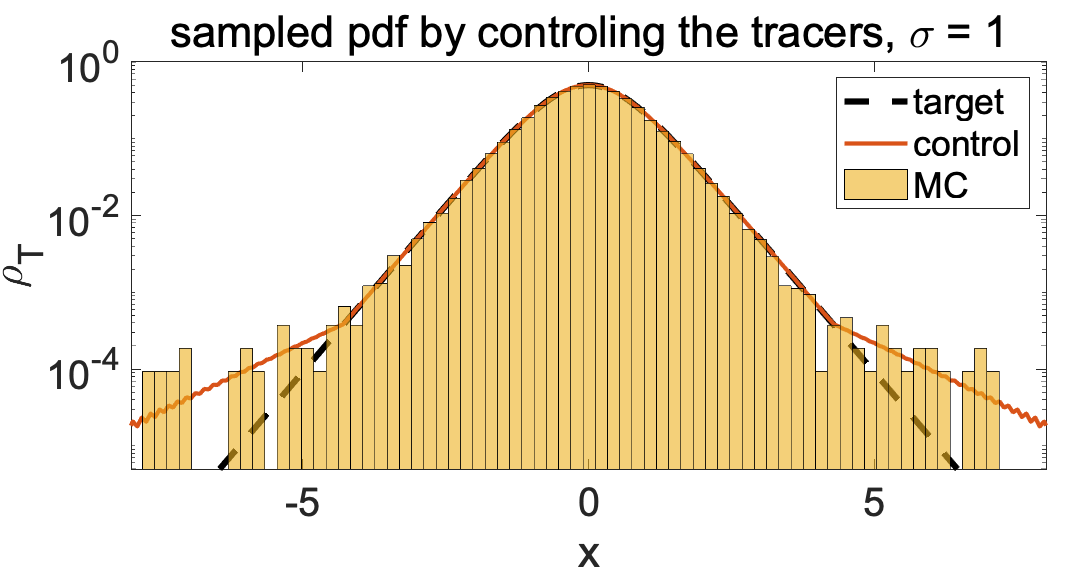}

\caption{Sampling non-Gaussian PDFs by controlling the tracer density in the SDE model.}\label{fig:tracer_sampling}
\end{figure}

\subsection{MFG-2 model for flow field control}

Next, we test the one-dimensional flow control problem formulated by the MFG-2 model. In this case, we need to solve the coupled joint equations \eqref{eq:mfg2}. 
Different from the previous tracer control problem, the continuity equation for $q_s$ becomes nonlinear due to the explicit coupling with the advection velocity $u_s=\mathcal{T}q_s$. The forward and backward equations become closely coupled. The new iterative approach 
in Algorithm \ref{alg:mfg2} is applied to solve the corresponding decoupled solution $\lbrace\tilde{q}_s,\tilde{\varphi}_s\rbrace$ from \eqref{eq:mfg2_modified} using  the solution  $\lbrace q^{(n)}_s,\varphi^{(n)}_s\rbrace$ from the previous iteration step
to find the converged optimal 
solution.

\subsubsection{Effective updating strategy during iterations}
We first confirm the necessary condition in Proposition~\ref{prop:stability} for converging iterations.  The new state $q^{\mu}=\mu q^n+\left(1-\mu\right)\tilde{q}$ is required to be updated using the optimal combination parameter $\mu\neq 0$. In Fig.~\ref{fig:cost}, we plot the improvement in minimizing the value function $\mathcal{I}\left(q^{\mu},\alpha^{\mu}\right)-\mathcal{I}\left(q^{(n)},\alpha^{(n)}\right)$ under the $L_2$ and KL-divergence cost during the first two iteration steps. Consistent with our analysis, the cost $\mathcal{I}\left(\tilde{q}_s,\tilde{\alpha}_s\right)$ as a function of $\mu$ gives approximately a quadratic structure with 0 at the right end $\mu=1$ and indefinite values on the left $\mu=0$. 
Notice that the cost function value by directly using the decoupled model solution $\lbrace\tilde{q}_s,\tilde{\varphi}_s\rbrace$ is indicated at $\mu=0$. The direct solution is not guaranteed to reduce the target cost function from each step of iteration. This is the inherent obstacle due to the instability in the iteration scheme especially during the initial iteration steps with larger errors. On the other hand, an improvement with negative values can be always achieved through a direct line searching of the optimal $\mu$ to minimize the cost. In addition, the choice of the optimal $\mu$ can also effectively expedite the convergence requiring only a few iteration steps.

As a more detailed illustration of the error development during iterations, Fig.~\ref{fig:iterations} plots the value function $\mathcal{I}(q_{s}^{\left(n\right)},\alpha_{s}^{\left(n\right)})$ as well as the $L_2$ errors in both states $q^{(n)}_s,\alpha^{(n)}_s$ measured during the entire control time window $s\in\left[0,T\right]$.
It shows that the value function converges quickly close to its minimum value just with a few iterations ($\sim 5$ under both $L_2$ and KL-divergence loss). In comparison, we also plot the evolution of the loss function values using a fixed $\mu=0.5$. Then, the value function cannot be minimized due to the frequent violation of the decreasing condition. This confirms the necessity of taking the adaptive choice of the combination parameter $\mu$ to reach the target fixed-point solution in the proposed iterative algorithm. We continue running the iterations further to the longer than necessary steps. It shows that the errors in the target states keep decreasing to the refined optimal solution during the iterations. This confirms the stability of the proposed method.

\begin{figure}
\centering
\includegraphics[scale=0.36]{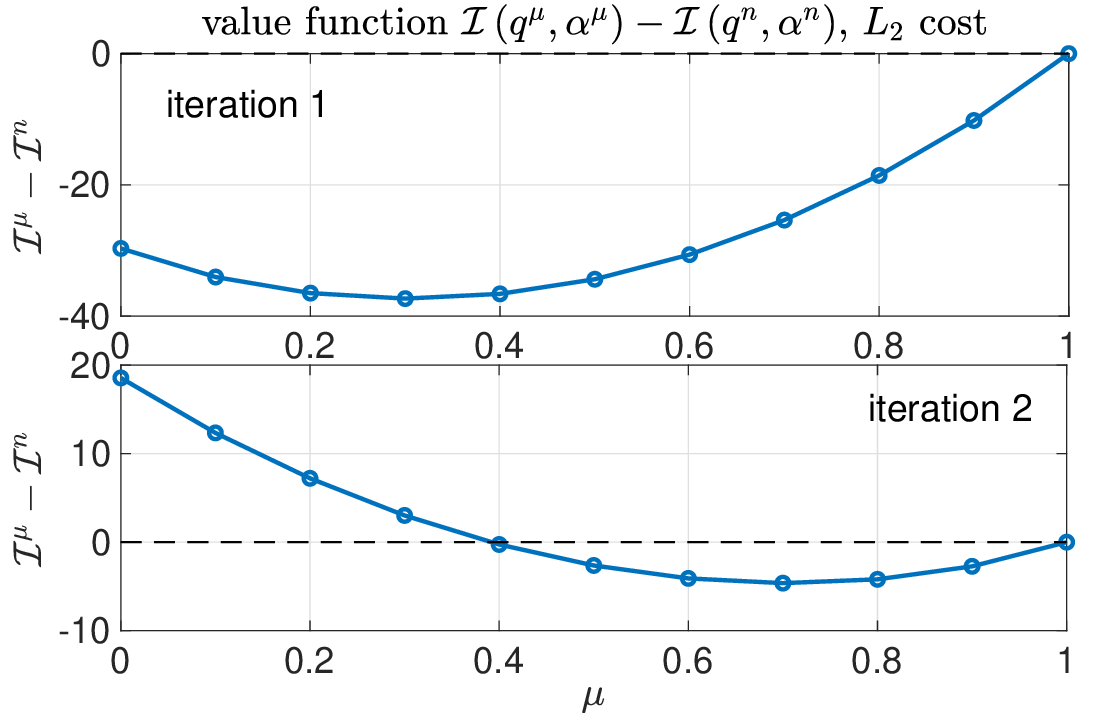}\includegraphics[scale=0.36]{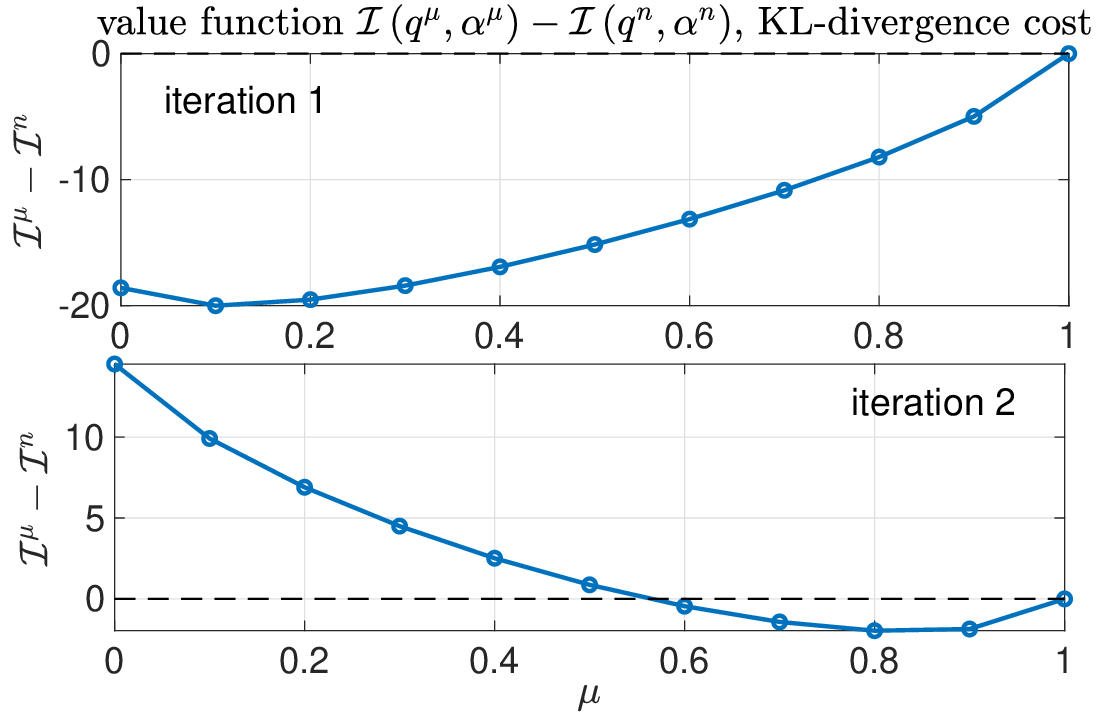}
\caption{Improvement in the target value function $\mathcal{I}\left(q^{\mu}_s,\alpha^{\mu}_s\right)-\mathcal{I}\left(q^{n}_s,\alpha^{n}_s\right)$ with different values of $\mu$ during the first two iterations using $L_2$ and KL-divergence cost.}\label{fig:cost}
\end{figure}
\begin{figure}
\centering
\subfloat{\includegraphics[scale=0.35]{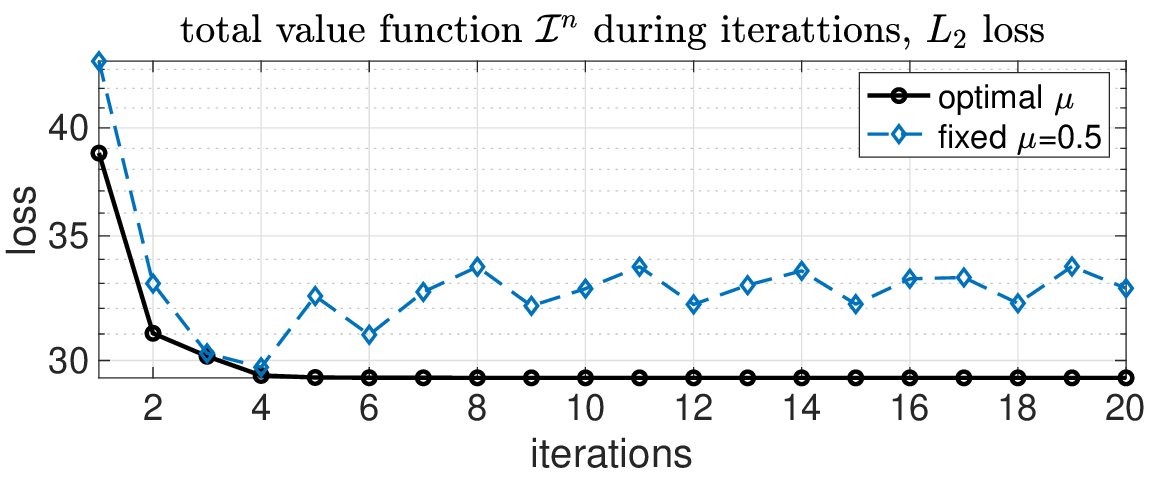}\includegraphics[scale=0.35]{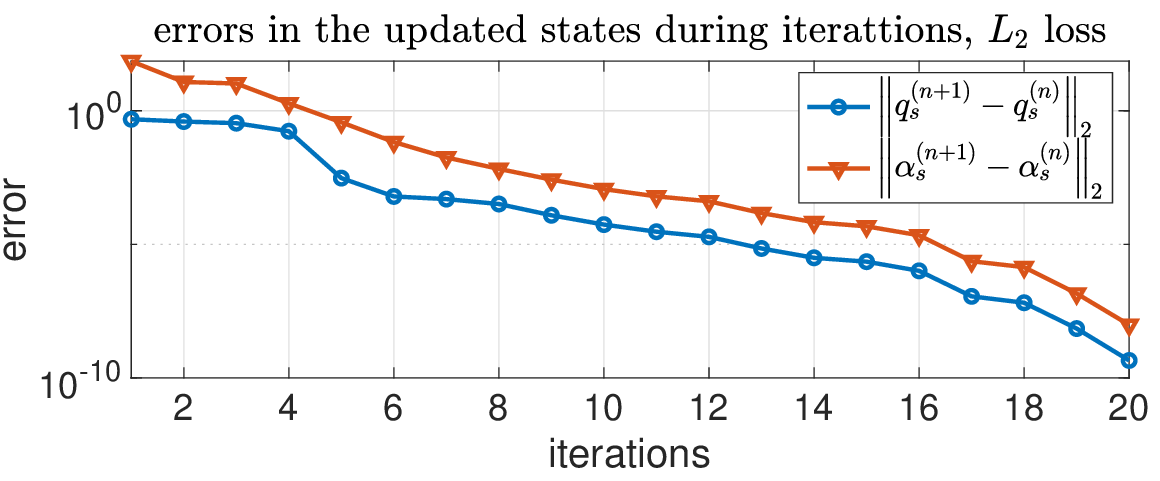} }
\vspace{-1em}
\subfloat{\includegraphics[scale=0.35]{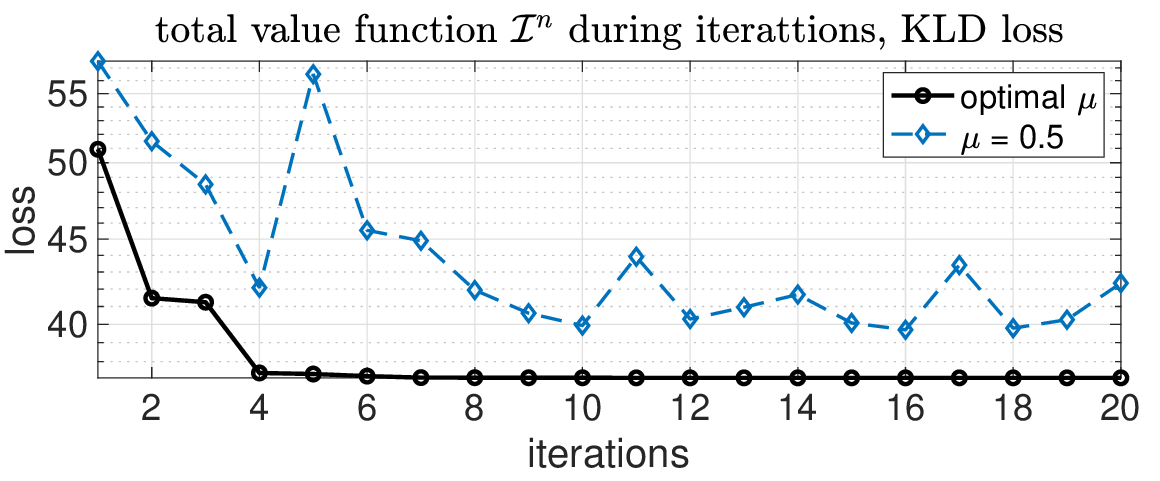}\includegraphics[scale=0.35]{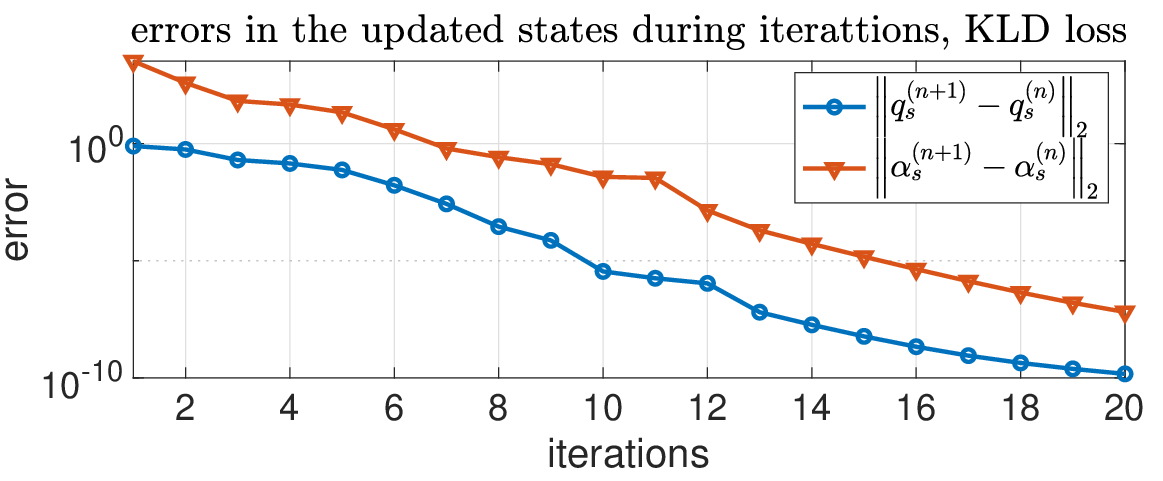} }

\caption{Total value function $\mathcal{I}^{n}=\mathcal{I}(q_{s}^{\left(n\right)},\alpha_{s}^{\left(n\right)})$
and the errors in the updated states $\lbrace q_{s}^{\left(n\right)},\alpha_{s}^{\left(n\right)}\rbrace $
during the updating iterations using both $L_{2}$ and KLD loss.}\label{fig:iterations}
\end{figure}

\subsubsection{Control performance with different cost functions}
Here, we show the optimally controlled solutions achieved from the iterative algorithm. The initial states for $\lbrace q_s^{(0)},\varphi_s^{(0)}\rbrace$ in the iterative algorithm is computed by solving the tracer control problem. The final converged optimal solutions using both $L_2$ and KL-divergence loss functions are displayed in Fig.~\ref{fig:control_flow1}. The controlled trajectories for $q_s$ show similar structure as in the tracer control case, while they are through the completely different dynamical equations. This can be seen by the very different structures in the optimal control solutions of $\alpha_s$. In fact, the solution illustrates the route of transition between the two steady states of the flow solutions.
\begin{figure}
\subfloat[$L_2$ loss]{
\includegraphics[scale=0.35]{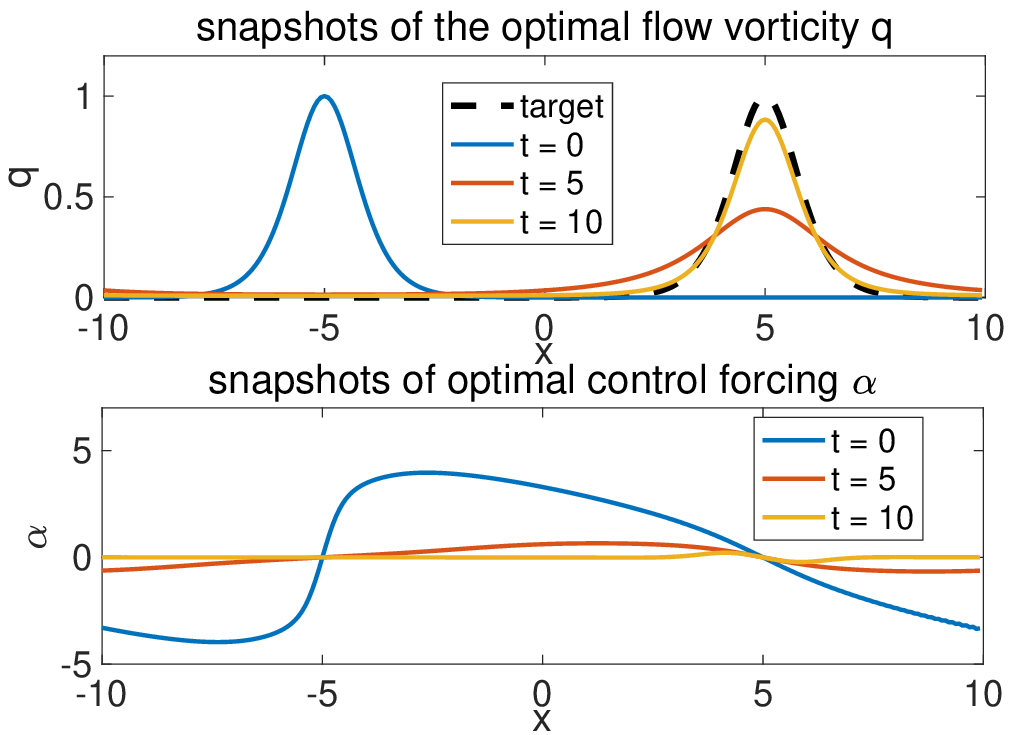}\includegraphics[scale=0.35]{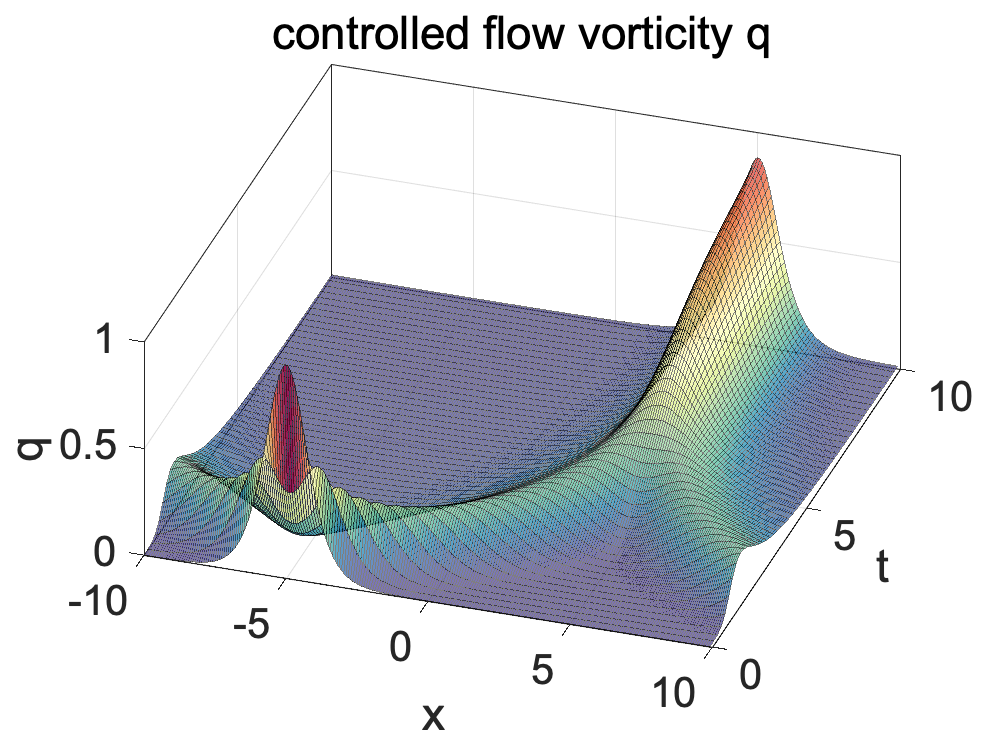}\includegraphics[scale=0.35]{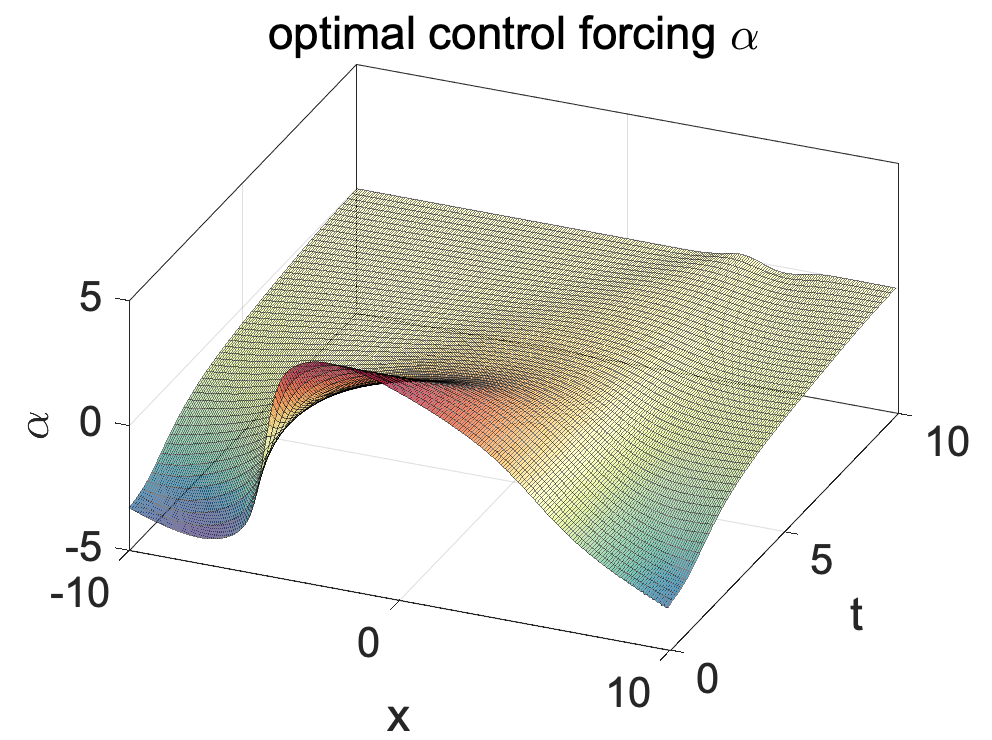}
}

\subfloat[KL-divergence loss]{
\includegraphics[scale=0.35]{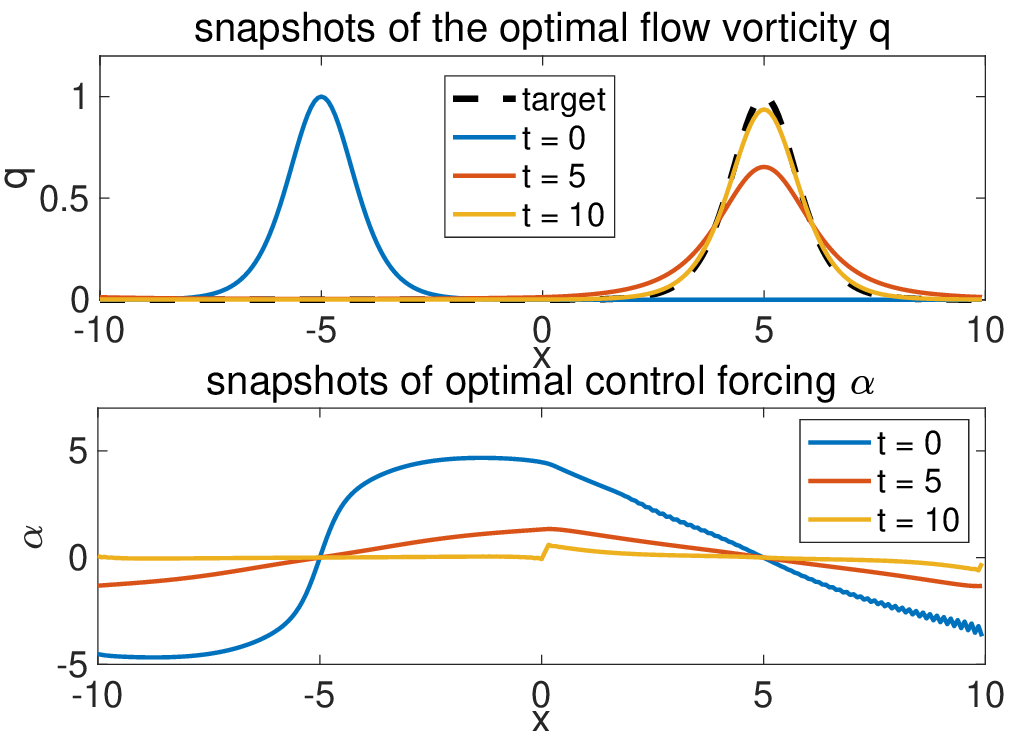}\includegraphics[scale=0.35]{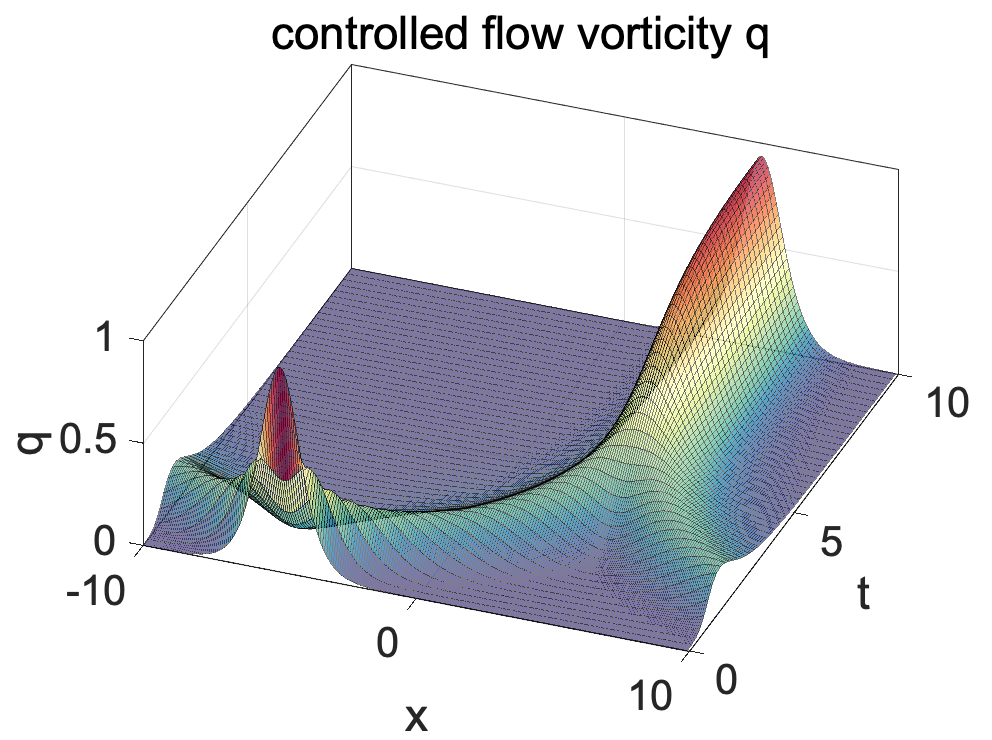}\includegraphics[scale=0.35]{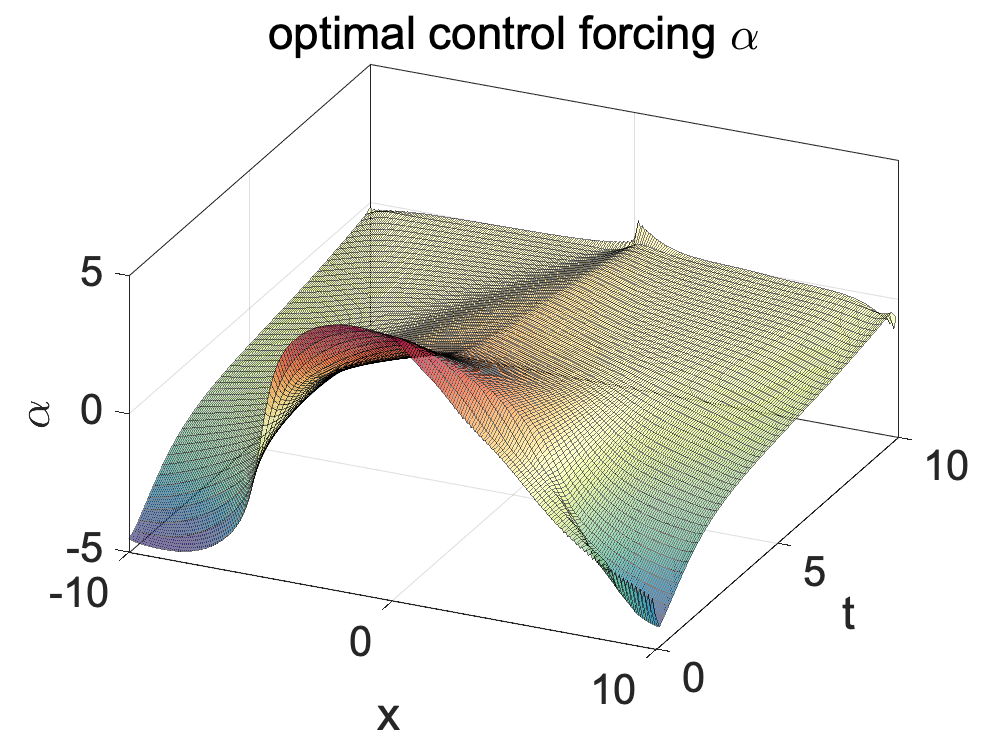}
}

\caption{Optimal controlled solution of MFG-2 model for vorticity transport with different loss functions.}\label{fig:control_flow1}
\end{figure}

Equivalently, we can solve the flow control problem using the SDE formulation \eqref{eq:loss_sde} with finite number of samples as described in the last formulation in Sec.~\ref{sec:control_models}. In this approach, instead of solving the continuous forward equation for $q_s$ we run an ensemble simulation for the Lagrangian tracers
\[
\mathrm{d}X_{s}^{i}  =\left[\mathcal{T}q_{s}^{N}\left(X_{s}^{i}\right)+\alpha_{s}^{N}\left(X_{s}^{i}\right)\right]\mathrm{d}s+\sqrt{2D}\mathrm{d}W_{s}^{i},
\]
with $i=1,\cdots,N$. Notice that the $N$ samples are linked together by the empirical recovery of the flow field from the samples $q_{s}^{N}\left(x\right)=\frac{1}{N}\sum_{i=1}^{N}\delta_{X_{s}^{i}}\left(x\right)$. The optimal control $\alpha^N_s$ is solved through the backward HJE also using the empirical estimate $q^N_s$. The initial samples agreeing with the initial state $Q_i$ are drawn by the strategy introduced in Sec.~\ref{subsec:sampling}. This idea is very useful for the generalization to higher dimensional cases where directly solving the forward Fokker-Planck PDE becomes highly expensive.
Fig.~\ref{fig:ctrl_mc} illustrates the controlled solutions using the KL-divergence loss at several time instants from the SDE model using $N=1\times10^4$ samples. The corresponding optimal PDE solution is Fig.~\ref{fig:control_flow1} is compared on top of the SDE solution from the empirical distribution. Good agreements are observed in the two equivalent approaches indicating effective skill in the control methods.

\begin{figure}
\centering
\includegraphics[scale=0.3]{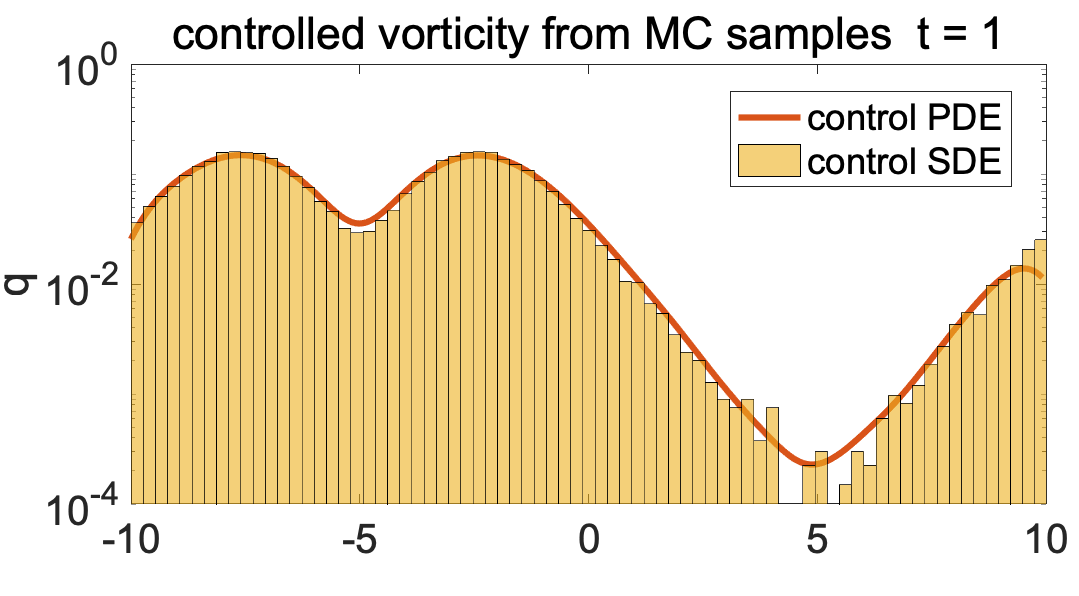}\includegraphics[scale=0.3]{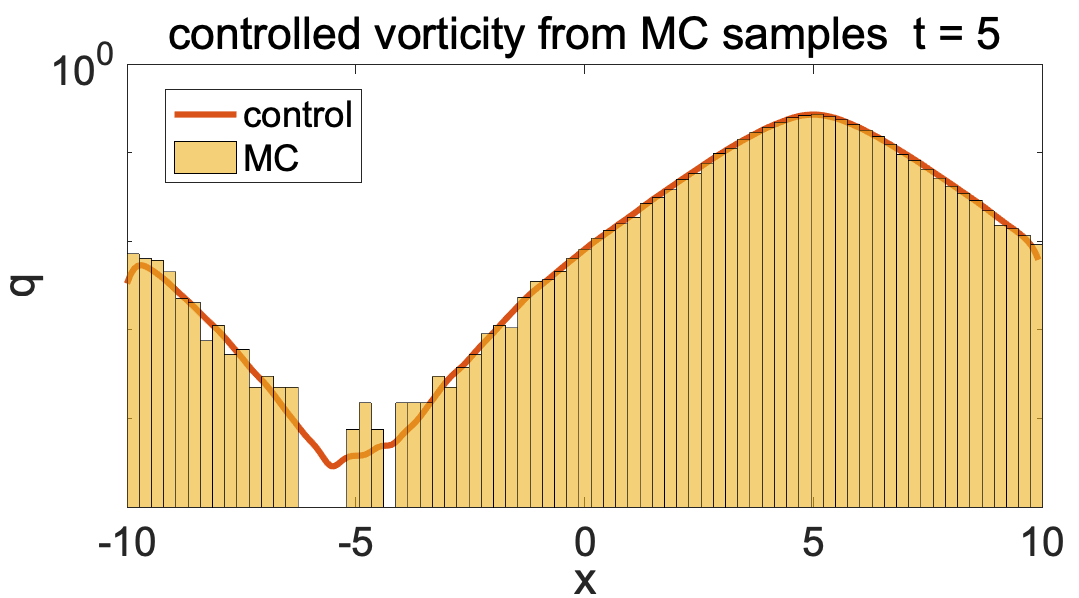}\includegraphics[scale=0.3]{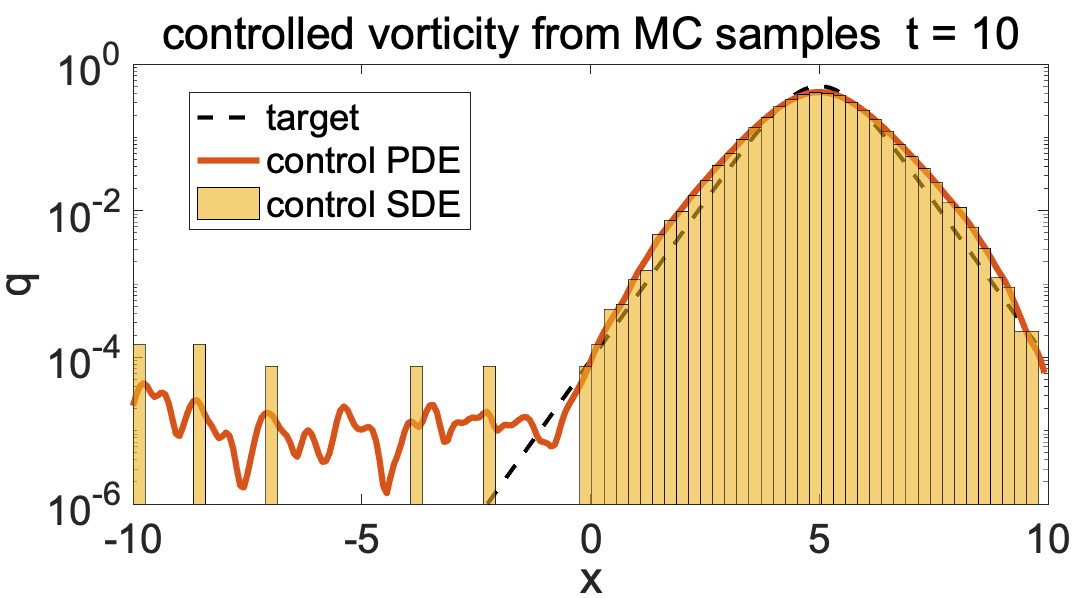}

\caption{Evolution of the controlled vorticity state $q_t$ from the SDE model using the KL-divergence loss function.}\label{fig:ctrl_mc}
\end{figure}

At last, we compare the optimal value function $\mathcal{V}\left(q,t\right)$ from different initial state $q=Q_{\sigma,-L/2}$ and starting time $t$. We fix the terminal time as $T=10$ and the shapes in the initial state is determined by the parameter $\sigma$. The results are shown in Fig.~\ref{fig:control_comp}. With $t=0$, the control has the longest time window $s\in\left[0,T\right]$. Thus the optimal solution is reached early and remain in the final target state to minimize the cost. When the starting time $t$ increases, a shorter time window $s\in\left[t,T\right]$ is allowed. Thus the control forcing needs to drive the solution to the final state in a faster rate. Overall, the optimal solutions go through a relatively similar transient stage.
\begin{figure}
\centering
\subfloat{
\includegraphics[scale=0.33]{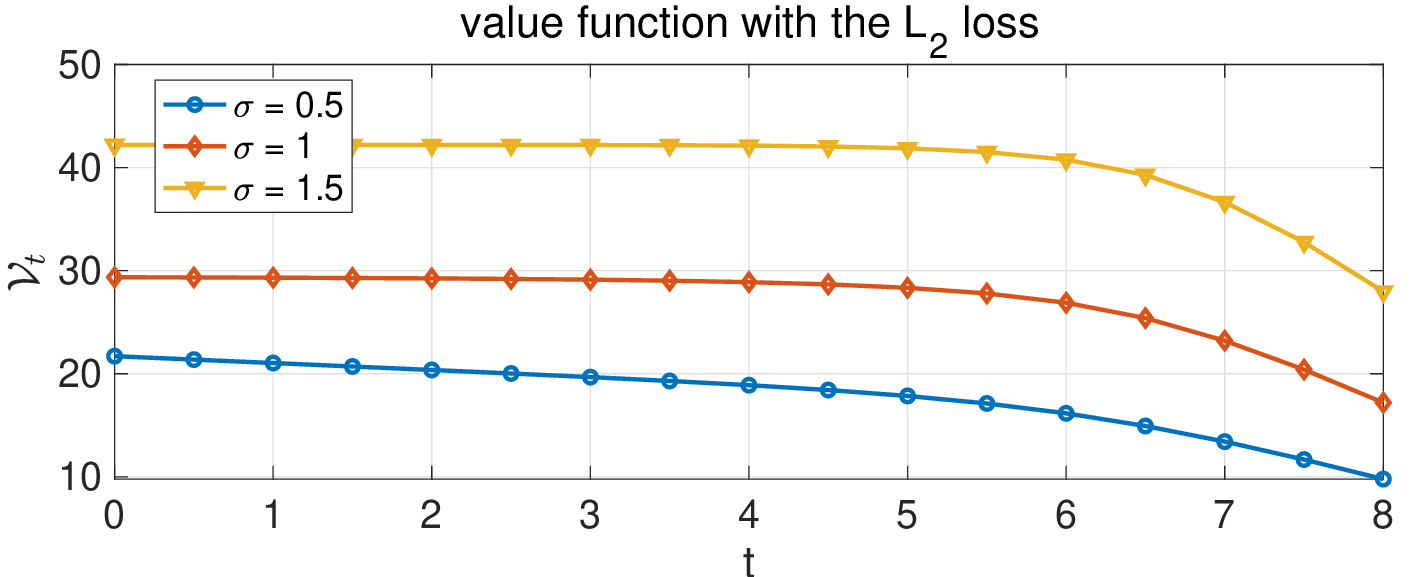}\includegraphics[scale=0.33]{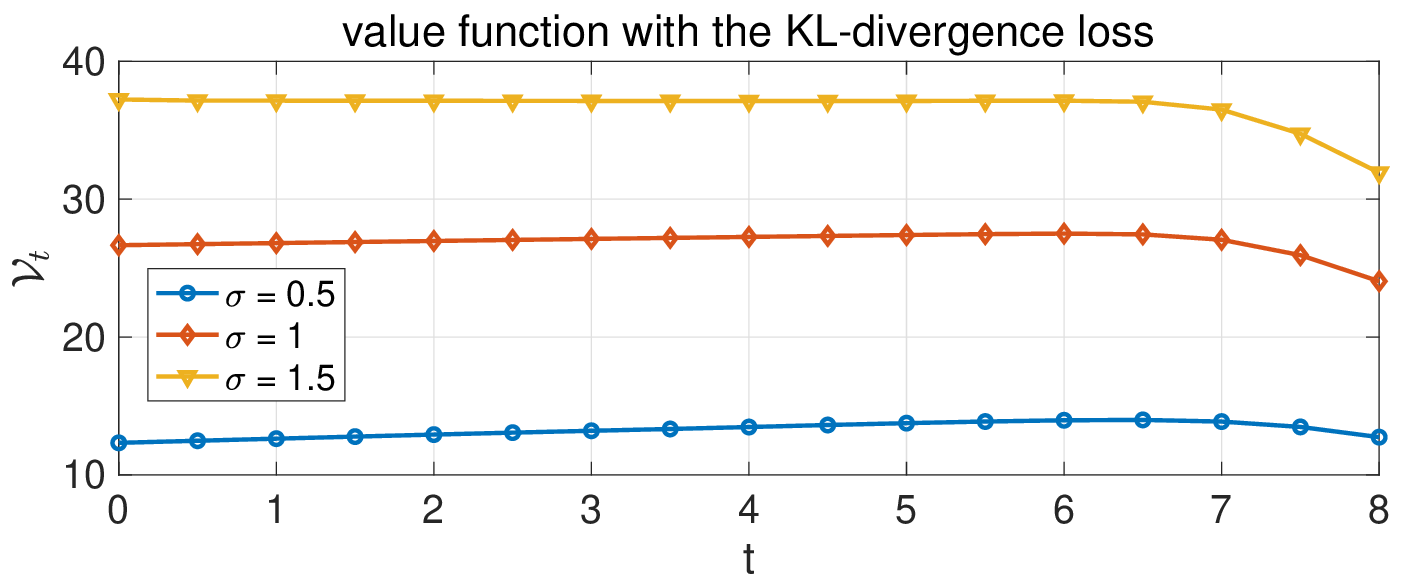}}
\vspace{-1em}
\subfloat{
\includegraphics[scale=0.34]{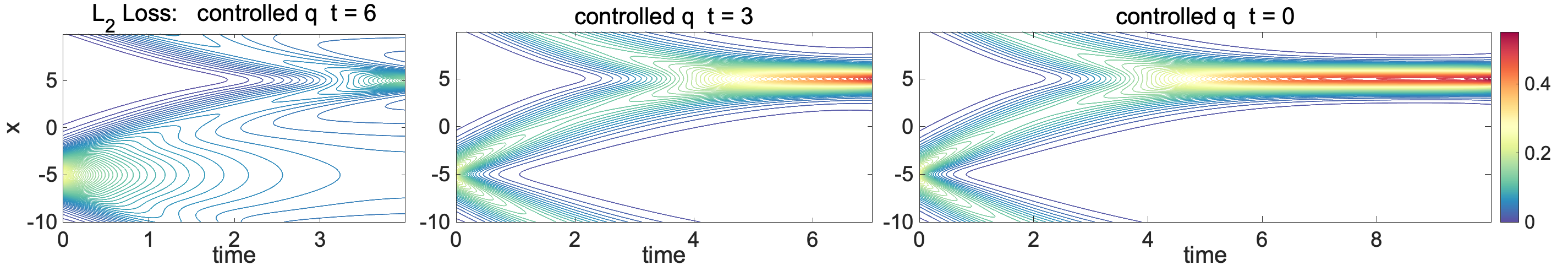}}
\vspace{-1em}
\subfloat{
\includegraphics[scale=0.34]{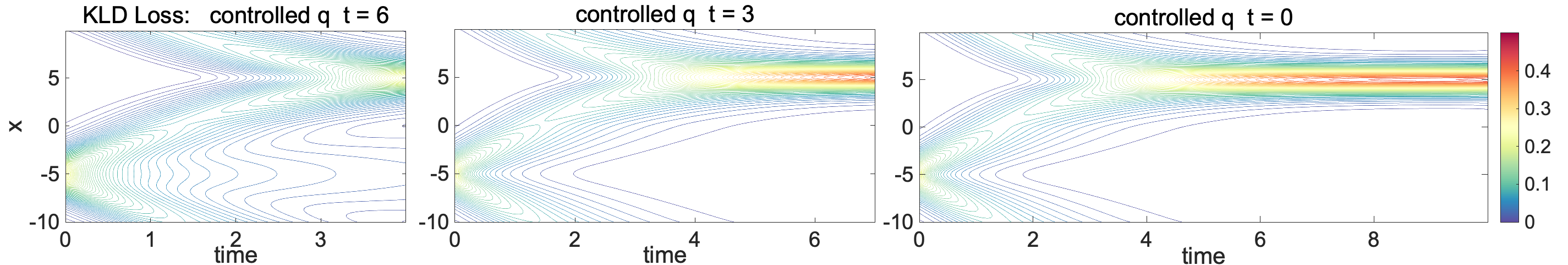}}

\caption{Comparison of the optimized value function $\mathcal{V}\left(q,t\right)$ with different starting time $t$ to the final time $T=10$. Different initial distributions $q=Q_{\sigma,-L/2}$ are also compared. The last two rows show the optimal controlled solutions $q_s$ with $\sigma=0.5$ from different starting times $t$.}\label{fig:control_comp}
\end{figure}

\subsection{MFG models for controlling two-dimensional flows}
Finally, we test the performance of the MFG models on controlling the 2-dimensional flow fields. The same algorithm is applied on solving the corresponding two-dimensional MVB equations in the same fashion. In this case, the equation becomes the diffusive transport of the vorticity field $q_s$.
The same steady solutions as in \eqref{eq:steady_mvb} are taken as the initial and final target states centered at $\left(-L/2,-L/2\right)$ and $\left(L/2,L/2\right)$. The same set of parameters is applied in the two-dimensional flow control case, and the same loss functions \eqref{eq:cost1} and \eqref{eq:cost2} are used measuring the errors in the two-dimensional functions.
Following the same strategy as in the one-dimensional tests, we first apply Algorithm~\ref{alg:mfg1} to get the initial guess $q_s^{(0)},\varphi_s^{(0)}$. Then the optimal control solution is achieved by the iterative strategy in Algorithm~\ref{alg:mfg2}. 

The the value function as well as the errors during each iteration are plotted in Fig.~\ref{fig:iterations_2d}. Still, we run a larger number of iterations to illustrate the evolution of the errors. The same as the one-dimensional case, the loss quickly converges to the minimum value after only a few iterations under both cost functions. The fast convergence is especially important in the two-dimensional case due to the much higher computational cost. The errors in the states $q_s^{(n)}$ and $\varphi_s^{(n)}$ also quickly drop to small values implying fast convergence and keep decreasing with more iterations just refining the detailed structures. 
The optimal trajectories for the optimal solution $q_s$  are shown in Fig.~\ref{fig:control_q_2d}. It is observed the recovered control forcing successfully moving the flow vorticity from the initial state to its final target. One interesting observation in the two-dimensional case is that the controlled flow solution demonstrates different routes to the target under the $L_2$ and KL-divergence loss functions. Under the $L_2$ loss, the solution goes through a gradual transition with decaying value in the initial state. On the other hand, under the KL-divergence loss, the initial density is moved directly to the final target along the four symmetric directions with the doubly periodic boundary.
Further future investigation is needed for a complete understanding of the distinctive behaviors under different loss functions.

\begin{figure}
\centering
\subfloat{\includegraphics[scale=0.35]{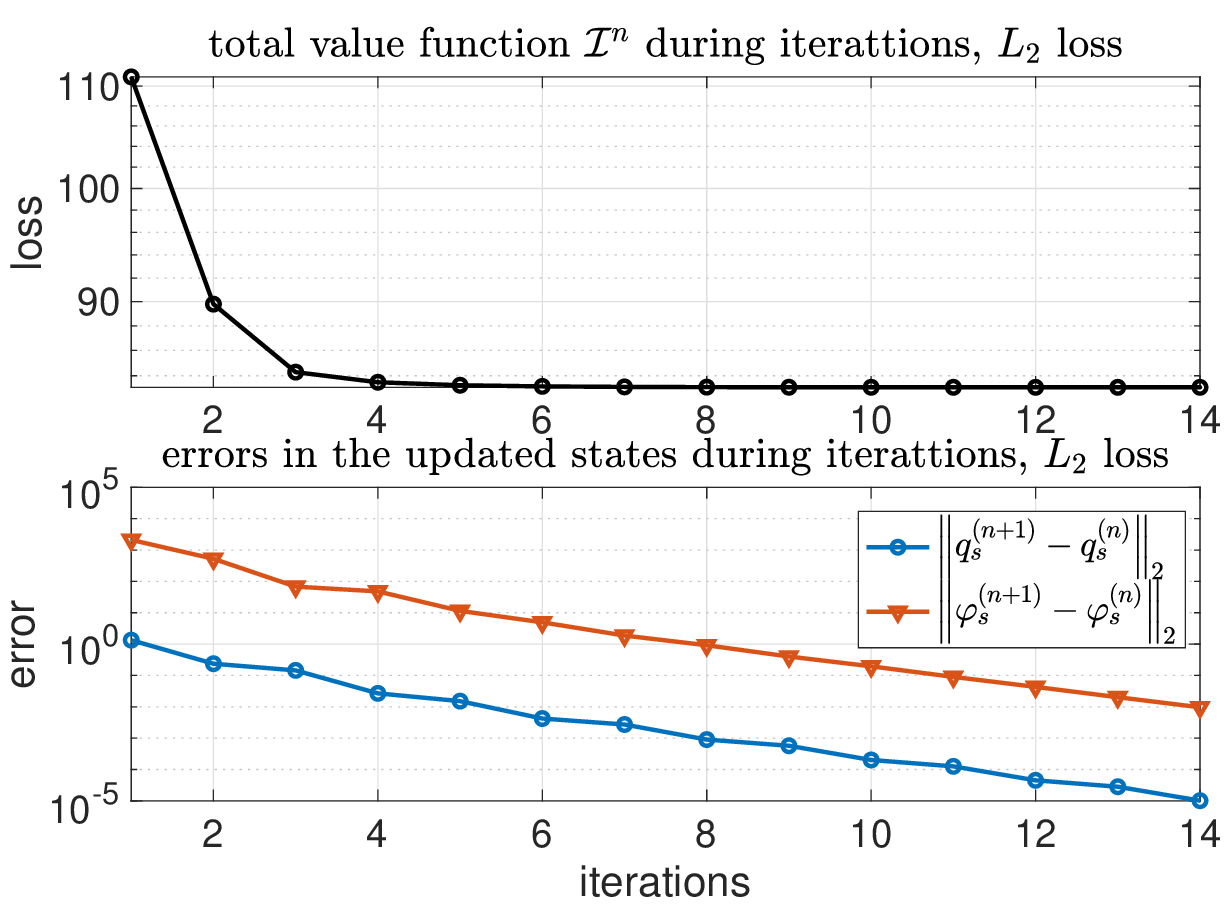} }
\subfloat{\includegraphics[scale=0.35]{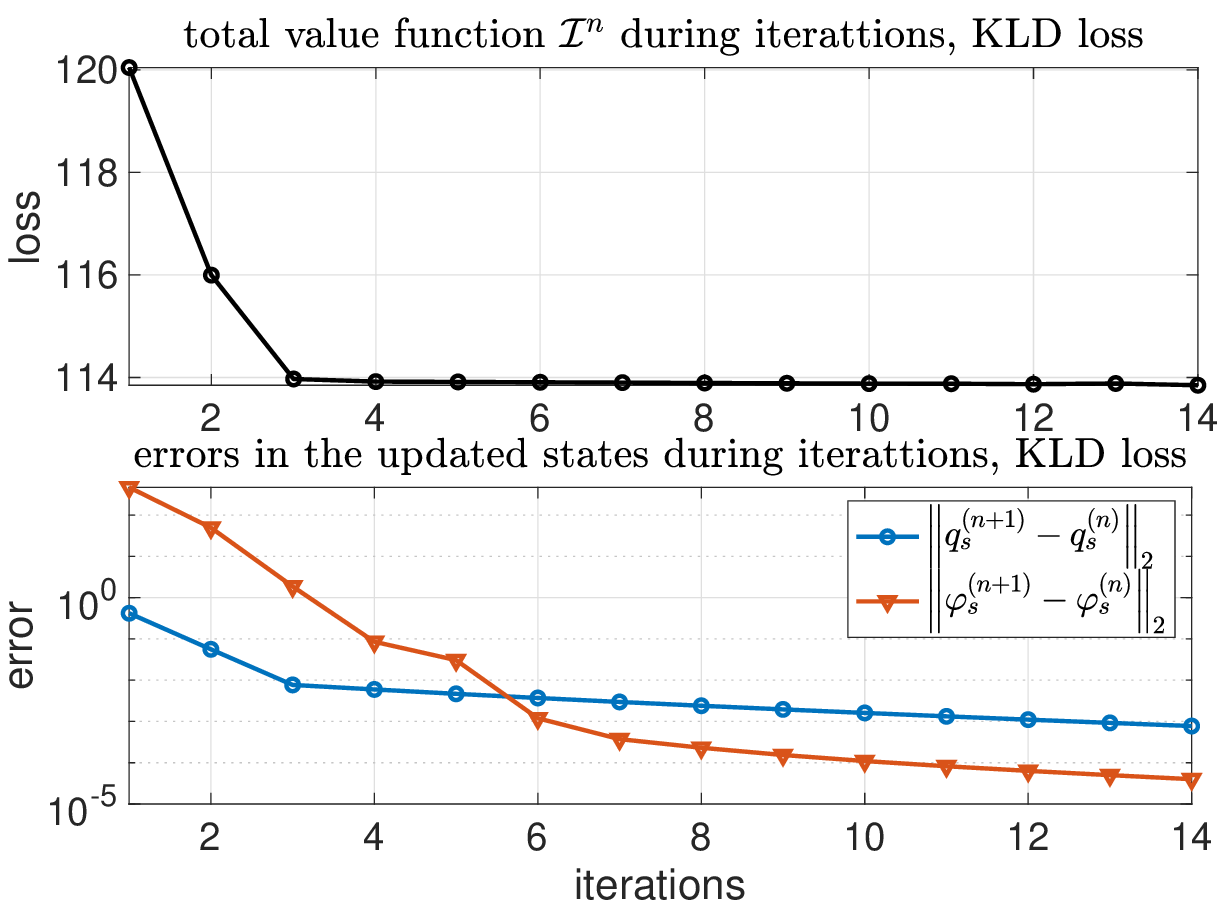} }

\caption{Total value function $\mathcal{I}^{n}=\mathcal{I}(q_{s}^{\left(n\right)},\alpha_{s}^{\left(n\right)})$
and the errors in the updated states $\lbrace q_{s}^{\left(n\right)},\varphi_{s}^{\left(n\right)}\rbrace $
during the updating iterations using both $L_{2}$ and KLD loss in the two-dimensional flow.}\label{fig:iterations_2d}
\end{figure}

\begin{figure}
\centering
\subfloat{
\includegraphics[scale=0.4]{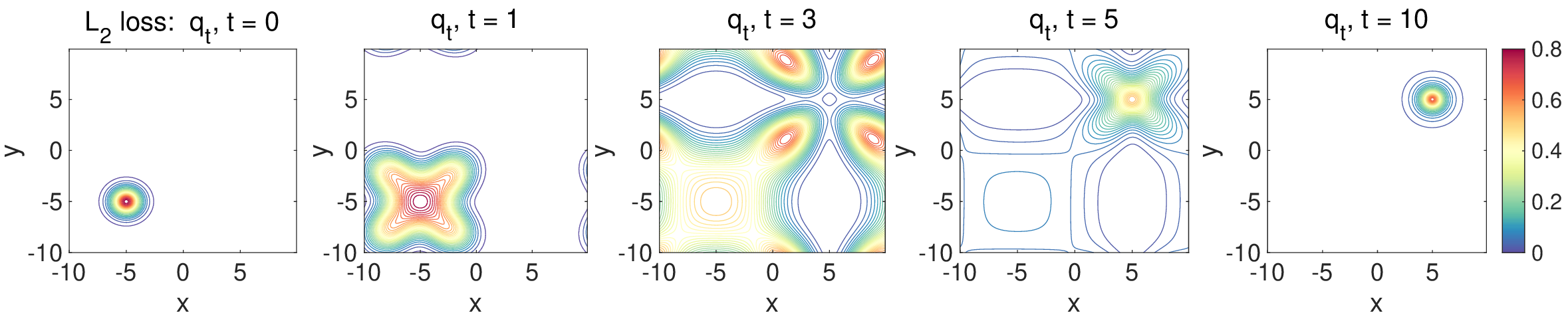}}
\vspace{-1em}
\addtocounter{subfigure}{-1}
\subfloat[optimal controlled solution $q_s$]{
\includegraphics[scale=0.4]{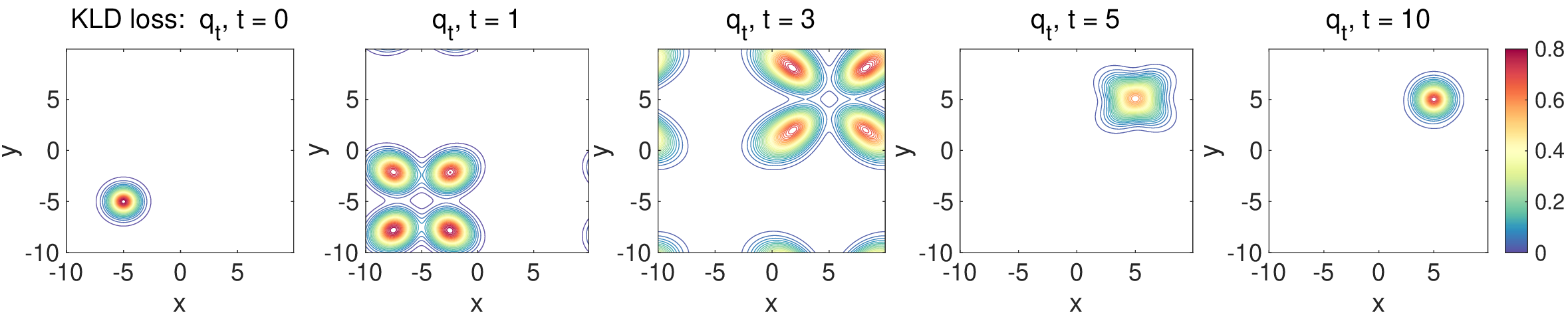}}

\subfloat{
\includegraphics[scale=0.4]{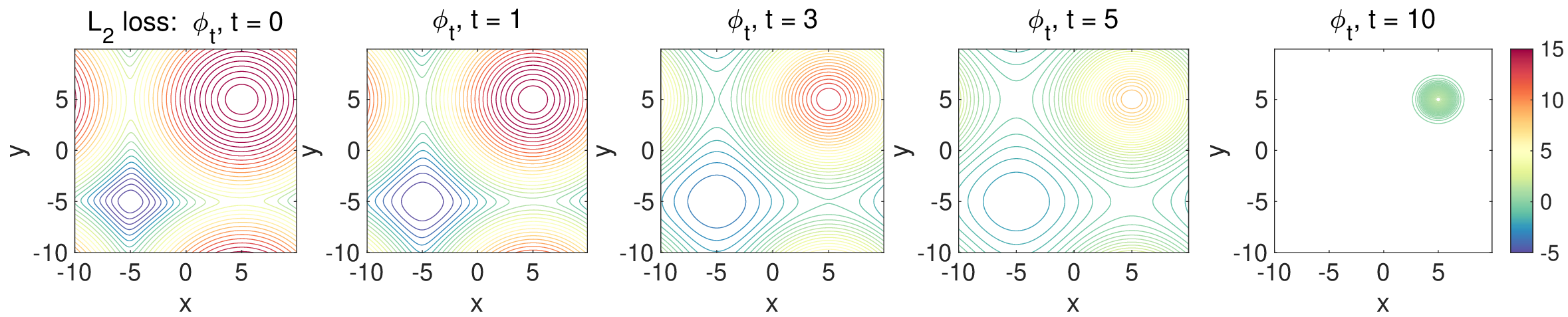}}
\vspace{-1em}
\addtocounter{subfigure}{-1}
\subfloat[optimal control function $\varphi_s$]{
\includegraphics[scale=0.4]{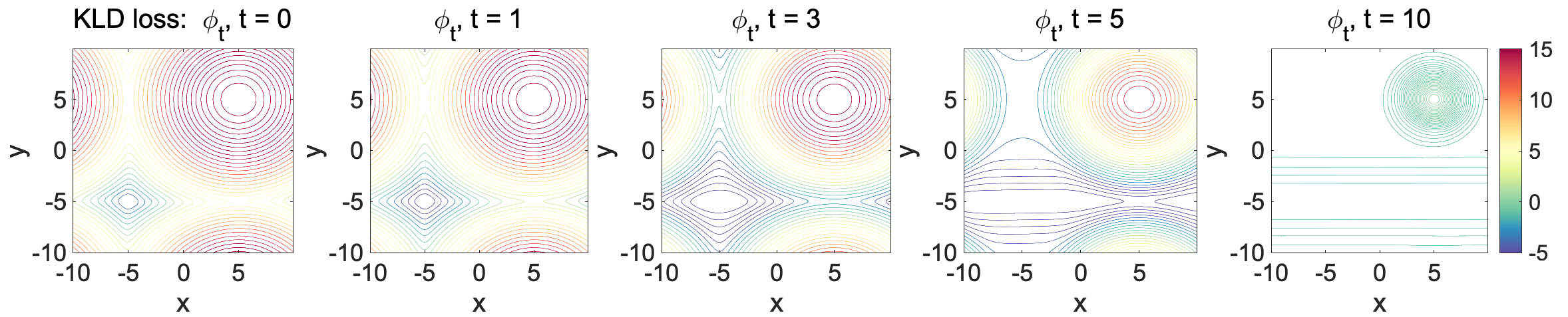}}
\caption{Optimal controlled solutions $q_s$ and control function $\varphi_s$ at several time instants $t$ under both the $L_2$ and KL-divergence loss functions.}\label{fig:control_q_2d}
\end{figure}

\section{Summary}\label{sec:summary}
In summary, we proposed mean field game models for controlling the nonlinear transport of tracer densities and flow vorticity fields under a unified mathematical framework. The mean field game models are solved through a forward continuity equation describing the tracer density/flow vorticity and a backward Hamilton-Jacobi equation for the optimal control action.
In addition, the general scalar vorticity field can be tracked by  Lagrangian tracers immersed in the advected flow field, thus an equivalent stochastic formulation can be derived providing a useful alternative approach for solving the controlled flow solution especially for the development of efficient numerical strategies in higher dimensional problems.
 The performance of the proposed MFG models and algorithms are then tested on the modified viscous Burger's equation displaying representative multiscaled and nonlinear dynamics.
Fast convergence and effective control performance are demonstrated in both the one- and two-dimensional test cases and under  loss functions in different metrics. 
For the future research, a more detailed convergence analysis for general systems is to be developed quantifying the approximation error and convergence rate of the new methods. The numerical strategy is also direct to apply the proposed algorithms to more general flow systems including stronger turbulent dynamics and instabilities.

\section*{Acknowledgments}
YG was partially supported by NSF under award DMS-2204288.
 

\bibliographystyle{plain}
\nocite{*}
\bibliography{refs}

\end{document}